\documentclass[oneside,english]{amsart}
\usepackage[T1]{fontenc}
\usepackage[latin9]{inputenc}
\usepackage{verbatim}
\usepackage{mathrsfs}
\usepackage{amstext}
\usepackage{amsthm}
\usepackage{amssymb}

\makeatletter
\numberwithin{equation}{section}
\numberwithin{figure}{section}
\theoremstyle{plain}
\newtheorem{thm}{\protect\theoremname}
\theoremstyle{plain}
\newtheorem{prop}[thm]{\protect\propositionname}
\theoremstyle{plain}
\newtheorem{lem}[thm]{\protect\lemmaname}
\theoremstyle{plain}
\newtheorem{cor}[thm]{\protect\corollaryname}

\usepackage[colorlinks=true,linkcolor=blue,urlcolor=black]{hyperref}

\makeatother

\usepackage{babel}
\providecommand{\corollaryname}{Corollary}
\providecommand{\lemmaname}{Lemma}
\providecommand{\propositionname}{Proposition}
\providecommand{\theoremname}{Theorem}

\begin{document}
\global\long\def\ZZ{\mathbb{Z}}%
 
\global\long\def\CC{\mathbb{C}}%
\global\long\def\HH{\mathbb{H}}%
\global\long\def\PP{\mathbb{P}}%
\global\long\def\QQ{\mathbb{Q}}%
\global\long\def\RR{\mathbb{R}}%
 
\global\long\def\NN{\mathbb{N}}%
 
\global\long\def\Hyp{\mathbb{H}^{2}}%

\global\long\def\Gne{\textbf{\ensuremath{\mathrm{{\mathbf{G}}}}}}%
\global\long\def\GLne{{\mathbf{GL}}}%
\global\long\def\SLne{{\mathbf{SL}}}%
\global\long\def\SOne{{\mathbf{SO}}}%

\global\long\def\spec{\text{Sp }}%

\global\long\def\Ac{\mathcal{A}}%
 
\global\long\def\Bc{\mathcal{B}}%
 
\global\long\def\Ec{\mathcal{E}}%
\global\long\def\Fc{\mathcal{F}}%
\global\long\def\Hc{\mathcal{H}}%
\global\long\def\Lc{\mathcal{L}}%
 
\global\long\def\Mc{\mathcal{M}}%
 
\global\long\def\Pc{\mathcal{P}}%
 
\global\long\def\Sc{\mathcal{S}}%
\global\long\def\Zc{\mathcal{Z}}%

\global\long\def\SL{\text{\ensuremath{\mathrm{SL}}}}%
\global\long\def\PSL{\mathrm{PSL}}%
\global\long\def\GL{\text{\ensuremath{\mathrm{GL}}}}%
\global\long\def\O{\text{\ensuremath{\mathrm{O}}}}%
\global\long\def\SO{\text{\ensuremath{\mathrm{SO}}}}%
\global\long\def\PSO{\mathrm{PSO}}%
 
\global\long\def\PO{\mathrm{PO}}%

\global\long\def\Scal{\mathscr{S}}%

\global\long\def\ne#1{\textbf{#1}}%

\global\long\def\eps{\varepsilon}%

\global\long\def\inv{^{-1}}%
 
\global\long\def\tr{\text{tr\,}}%
\global\long\def\im{\text{Im }}%
 
\global\long\def\re{\text{Re }}%
 
\global\long\def\cl{\text{cl\,}}%

\global\long\def\gexp{\delta_{\Gamma}}%
 
\global\long\def\llG{\ll_{\Gamma}}%
 
\global\long\def\ggG{\gg_{\Gamma}}%
 
\global\long\def\asymG{\asymp_{\Gamma}}%

\global\long\def\Words{\mathcal{W}}%
\global\long\def\neWords{\mathcal{\mathcal{W}^{\circ}}}%
 
\global\long\def\mirror#1{\widetilde{#1}}%
\global\long\def\Start#1{S(#1)}%
 
\global\long\def\End#1{E(#1)}%
 
\global\long\def\wlength#1{|#1|}%

\global\long\def\Pp{\Pc_{p}^{*}}%
\global\long\def\Pn{\Pc_{n}^{*}}%

\global\long\def\len#1{|I_{#1}|}%
 
\global\long\def\norm#1{\left|\left|#1\right|\right|}%
 
\global\long\def\trNorm#1{||#1||_{\text{tr}}}%
 
\global\long\def\HSnorm#1{||#1||_{{\scriptscriptstyle HS}}}%

\global\long\def\scal#1#2{\left\langle #1,#2\right\rangle }%

\title{Explicit spectral gap for Schottky subgroups of $\SL(2,\ZZ)$}
\author{Irving Calderón}
\author{Michael Magee}
\begin{abstract}
Let $\Gamma$ be a Schottky subgroup of $\SL(2,\ZZ)$. We establish
a uniform and explicit lower bound of the second eigenvalue of the
Laplace-Beltrami operator of congruence coverings of the hyperbolic
surface $\Gamma\backslash\Hyp$ provided the limit set of $\Gamma$
is thick enough. 
\end{abstract}

\maketitle

\section{Introduction}

\subsection{Spectral gaps for hyperbolic surfaces and main result}

The Laplace-Beltrami operator $\Delta_{S}$ of a hyperbolic surface
$S=\Gamma\backslash\Hyp$ encodes important geometric and dynamical
features of $S$. The spectral theory of $\Delta_{S}$ also has deep
connections with number theory when $\Gamma$ is an arithmetic subgroup
of $\SL(2,\RR)$. We are interested in the general problem of giving
lower bounds of the second smallest eigenvalue $\lambda_{1}(S)$ of
$\Delta_{S}$ when $S$ varies in a family $\Fc$ of hyperbolic surfaces.
We call these \emph{spectral gap results}. In this work, $\mathcal{F}$
will always be a family of finite covers of a fixed surface $S$.
Not all such families have a spectral gap. For example, suppose that
$S=\Gamma\backslash\Hyp$ for some $\Gamma$ admitting a surjective
homomorphism $\psi:\Gamma\to\ZZ$. The $\lambda_{1}$ of the finite
cover $S_{(n)}:=\psi\inv(n\ZZ)\backslash\Hyp$ of $S$ tends to $0$
as $n\to\infty$. It is believed that this kind of examples are rare,
and that most finite coverings of a hyperbolic surface have a strong
spectral gap. For instance, this was recently established in \cite{magee_explicit_2020}
and \cite{magee-naud_extension_2021} for random coverings of Schottky
surfaces.

Many $\Fc$ of arithmetic flavor do have a spectral gap, even a uniform
one. By an arithmetic $\Fc$ we mean the following: suppose $\Gamma$
is a subgroup of $\SL(2,\ZZ)$, let $\psi_{n}$ be the projection
$\SL(2,\ZZ)\to\SL(2,\ZZ/n\ZZ)$, and let $G_{n}$ be a subgroup of
$\SL(2,\ZZ/n\ZZ)$. We take $\Fc$ as the hyperbolic surfaces associated
to $\psi_{n}\inv(G_{n})\cap\Gamma$. We will focus on the important
case where all the $G_{n}$'s are $0$. We call $\Gamma\cap\psi_{n}\inv(0)$
the \emph{$n$-th congruence subgroup} of $\Gamma$---which we denote
$\Gamma_{n}$---. Similarly, $S_{n}:=\Gamma_{n}\backslash\Hyp$ is
the \emph{$n$-th congruence covering} of $S=\Gamma\backslash\Hyp$.
More generally, any $\Gamma$ contained in an arithmetic subgroup
of $\SL(2,\RR)$ has congruence subgroups. 

Our main result is a uniform spectral gap for congruence coverings
of $\Gamma\backslash\Hyp$, for any Schottky subgroup $\Gamma$ of
$\SL(2,\ZZ)$ with big enough growth exponent---see Subsection \ref{subsec:Basic_definitions}
for the base definitions and facts of hyperbolic surfaces and their
Laplace-Beltrami operators, and Subsection \ref{subsec:Fuchsian-Schottky-groups}
for the definition of Schottky group---. Before stating it, let us
give context on uniform spectral gaps for congruence coverings of
hyperbolic surfaces. In the seminal paper \cite{selberg_estimation_1965},
Selberg was interested in the family $\Fc$ of congruence coverings
of $\SL(2,\ZZ)\backslash\Hyp$\footnote{The quotient $\SL(2,\ZZ)\backslash\Hyp$ is actually not a surface
because $\PSL(2,\ZZ)$ has torsion, but all the other elements of
$\Fc$ are surfaces.} . His motivation came from number theory and, more precisely, modular
forms. Selberg showed that $\lambda_{1}(S)\geq\frac{3}{16}$ for any
$S\in\Fc$. This result is commonly called Selberg's $\frac{3}{16}$-Theorem.
For cocompact arithmetic subgroups\footnote{See \cite[p. 210]{sarnak_bounds_1991} for the precise definition.}
$\Gamma$ of $\SL(2,\RR)$, Sarnak and Xue show in \cite[Corollary 2]{sarnak_bounds_1991}
that any prime congruence covering $S'$ of $S=\Gamma\backslash\Hyp$
of high enough degree verifies

\[
\lambda_{1}(S')\geq\min\left\{ \frac{5}{36},\lambda_{1}(S)\right\} .
\]

Consider now $S=\Gamma\backslash\Hyp$ with $\Gamma$ a subgroup of
$\SL(2,\ZZ)$. We denote by $m_{\Gamma}(n,J)$ the number of eigenvalues---counted
with multiplicity---of $S_{n}$ in the interval $J$. We will write
$m_{\Gamma}(J)$ instead of $m_{\Gamma}(1,J)$. Assume also that $\Gamma$
is finitely generated so that $\Delta_{S}$ has finitely many eigenvalues
$\lambda_{0}(S)\leq\cdots\leq\lambda_{k}(S)$ in $[0,1/4)$---see
Subsection \ref{subsec:Basic_definitions}---. Bourgain, Gamburd
and Sarnak show in \cite[Theorem 1.1]{bourgain-gamburd-sarnak_generalization_2011}
that the $S_{q}$ with $q$ square-free have a uniform spectral gap
whenever its growth exponent $\gexp$ is $>\frac{1}{2}$, although
they do not give an explicit lower bound of $\lambda_{1}$. Oh and
Winter establish the analogous result when $\delta_{\Gamma}\leq\frac{1}{2}$
in \cite{oh-winter_uniform_2015}. This is extended to arbitrary moduli
in \cite{magee-oh-winter_uniform_2019}. In this case, the Laplace-Beltrami
operator does not have eigenvalues, so the spectral gap is formulated
in the terms of \emph{resonances }of $S_{q}$. These are poles of
the meromorphic continuation of the resolvent $R_{S_{q}}(s)=(\Delta_{S_{q}}-s(1-s)I)\inv$
of $\Delta_{S_{q}}$. 

An explicit spectral gap in this setting is \cite[Main Theorem]{gamburd_spectral_2002},
an earlier result of Gamburd which we restate below. Note that the
price to pay for the explicit gap is a more restrictive lower bound
on $\gexp$. 
\begin{thm}
\label{thm:Gamburds_thm}Consider a hyperbolic surface $S=\Gamma\backslash\Hyp$
with $\Gamma$ a finitely generated subgroup of $\SL(2,\ZZ)$, and
let $J=[0,\frac{5}{36})$. If $\gexp>\frac{5}{6}$, then $m_{\Gamma}(p,J)=m_{\Gamma}(J)$
for any big enough prime number $p$.
\end{thm}

Our main result improves Gamburd's theorem---although for a more
restricted class of surfaces---in two ways: we relax the condition
on $\gexp$ and we get a bigger interval without new eigenvalues.
In the statement we use the notation $t_{\Gamma}=\frac{\gexp}{6}+\frac{2}{3}$
and $J_{\beta}=[0,\beta(1-\beta)]$.
\begin{thm}
\label{thm:MainThm}Consider a hyperbolic surface $S=\Gamma\backslash\Hyp$,
with $\Gamma$ a Schottky subgroup of $\SL(2,\ZZ)$. If $\gexp>\frac{4}{5}$,
then for any $\beta\in\left(t_{\Gamma},\gexp\right)$ there is a constant
$C(\Gamma,\beta)$ with the next property: for any integer $n\geq5$
such that all its prime divisors are $\geq C(\Gamma,\beta)$ we have
\[
m_{\Gamma}(n,J_{\beta})=m_{\Gamma}(J_{\beta}).
\]
\end{thm}

Note that we deal with general moduli $n$ in contrast to many previous
works.

\subsection{Strategy of proof}

To prove Theorem \ref{thm:MainThm} we will show on the one hand that---for
any big enough $n$---$m_{\Gamma}(n,J_{\beta})$ is $\leq\mathtt{D}n^{1-\eps}$
for a constant $\mathtt{D}$ depending only on $\Gamma$ and $\beta$,
and some $\eps>0$. On the other hand, $m_{\Gamma}(n,J_{\beta})\gg n/3^{\omega(n)}$
when $S_{n}$ has a new eigenvalue in $J_{\beta}$. The two inequalities
cannot hold at the same time for big $n$, so Theorem \ref{thm:MainThm}
follows. This strategy to prove uniform spectral gaps is due to Sarnak
and Xue \cite{sarnak_bounds_1991}\footnote{It seems to be common knowledge that the strategy is based on an idea
of Kazdhan.}. 

The lower bound of $m_{\Gamma}(n,J_{\beta})$ is standard and relies
on the representation theory of the finite group $\Gamma/\Gamma_{n}$,
which is isomorphic to $\SL(2,\ZZ/n\ZZ)$ for any $n$ without prime
divisors in a finite set depending on $\Gamma$. The novelty of our
work is the following upper bound of $m_{\Gamma}(n,J_{\beta})$. Although
our methods work only for $\beta>t_{\Gamma}$, we think the bound
should also hold for smaller $\beta$. 
\global\long\def\consMultiplicityp{\mathtt{C}_{\Gamma}}%
 
\global\long\def\consMultiplicityCoeff#1{\mathtt{D}_{\Gamma,#1}}%
 
\global\long\def\consMultiplicityExponent#1{\xi(\gexp,#1)}%

\begin{prop}
\label{prop:Multiplicity_bound}Let $\Gamma$ be a Schottky subgroup
of $\SL(2,\ZZ)$ with $\gexp>\frac{4}{5}$ and consider $\beta\in(t_{\Gamma},\gexp)$.
There are constants $\consMultiplicityp>0,\consMultiplicityCoeff{\beta}>0$
and $\xi=\consMultiplicityExponent{\beta}\in(0,1)$ with the next
property: For any integer $n>\consMultiplicityp$ we have 
\[
m_{\Gamma}(n,J_{\beta})\leq\consMultiplicityCoeff{\beta}n^{\xi}.
\]
\end{prop}

Let us close by highlighting the key points of our proof of Proposition
\ref{prop:Multiplicity_bound}. Unlike Sarnak-Xue \cite{sarnak_bounds_1991}
and Gamburd \cite{gamburd_spectral_2002}, who rely on Selberg's Trace
Formula to get an upper bound of $m_{\Gamma}(p,J_{\beta})$, we exploit
the fact that the eigenvalues of a Schottky surface $S$ are zeros
of certain dynamical zeta functions $\zeta$---which are entire functions
on $\CC$---attached to $S$. We bound thus the number of eigenvalues
of $S$ in an interval by the number of zeros of a convenient $\zeta$
in some domain of $\CC$. 

As explained in Subsection \ref{subsec:Fuchsian-Schottky-groups},
any Schottky subgroup $\Gamma$ of $\SL(2,\RR)$ comes with a finite
generating set $\mathscr{G}$ and a distinguished fundamental domain
$\Fc$ in $\Hyp$. Intuitively, the zeta functions of $S=\Gamma\backslash\Hyp$
come from nonbacktracking walks in $\Hyp$ starting form $\Fc$. The
classical case is the walk $\mathscr{W}_{1}$ with steps of length
one: a path is an infinite sequence $(x_{0},x_{1},\ldots)$ starting
from $\Fc$, such that $x_{n+1}=\gamma x_{n}$ for some $\gamma\in\mathscr{G}$,
and $x_{n+1}\neq x_{n-1}$. Other zeta functions are obtained by walking
faster, for example with steps of length $m$, or length $m_{i}$
when moving in a certain ``direction'' $C_{i}$. The zeta functions
$\zeta_{\tau,n}$ we use to prove Proposition \ref{prop:Multiplicity_bound}---introduced
by Bourgain and Dyatlov in \cite{bourgain_fourier_2017}---come from
walks with steps of \emph{length at infinity $\tau$}. The key point
is choosing the $\tau$ giving the best upper bound for $m_{\Gamma}(n,J_{\beta})$.

\subsection{Arithmetic applications of spectral gaps}

Here we explain how our main result can improve some arithmetic results
obtained by the affine linear sieve of Bourgain, Gamburd and Sarnak
\cite{bourgain-gamburd-sarnak_affine_2010}. First we recall the general
problem addressed by the affine linear sieve and then we present a
concrete example.

Consider a finitely generated subgroup $\Gamma$ of $\GL(d,\ZZ)$,
a rational representation $\rho:\GLne(d)\to\GLne(m)$, the $\Gamma$-orbit
$\mathcal{O}=v_{0}\rho(\Gamma)$ of some $v_{0}\in\QQ^{m}$ and a
polynomial $P(x)\in\QQ[x_{1},\ldots,x_{m}]$ taking integral values
at $\mathcal{O}$. Loosely speaking, the affine linear sieve gives
conditions on $\Gamma,\mathcal{O}$ and $P$ which guarantee that
$P$ takes values with few prime divisors in a big subset of $\mathcal{O}$.
Dirichlet's Theorem on primes in arithmetic progressions---there
are infinitely many prime numbers of the form $a+nb$ for any relatively
prime, nonzero integers $a,b$---can be formulated in this way taking

\[
\Gamma=\left\langle \begin{pmatrix}1 & 0\\
b & 1
\end{pmatrix}\right\rangle ,\mathcal{O}=(a,1)\Gamma,\quad\text{and}\quad P(x_{1},x_{2})=x_{1}.
\]
We need some definitions to precise what we mean by \emph{big subset}
of $\mathcal{O}$ in general. An \emph{$R$-almost prime} is an integer
with at most $R$ prime divisors, counted with multiplicity. For example,
$-12$ is a 3-almost prime. The subsets 

\[
\mathcal{O}_{P}(R)=\left\{ x\in\mathcal{O}\mid P(x)\text{ is an }R\text{-almost prime}\right\} 
\]
of $\mathcal{O}$ grow with $R$. We say that $(\mathcal{O},P)$ \emph{saturates
}if $\mathcal{O}_{P}(R)$ is Zariski-dense in $\mathcal{O}$ for some
$R$. In that case, the minimal such $R$ is the \emph{saturation
number }of $(\mathcal{O},P)$, and is denoted $R(\mathcal{O},P)$.
The problem we consider is how to show that $(\mathcal{O},P)$ saturates.
Better yet, can we give an upper bound of $R(\mathcal{O},P)$? The
affine linear sieve, and more precisely \cite[Theorem 1.1]{bourgain-gamburd-sarnak_generalization_2011},
answers these questions for various $(\mathcal{O},P)$. This result
treats the case where $\rho:\GLne(d)\to\GLne(d^{2})$ is the representation
of $\GL(d)$ on the space of $d\times d$ matrices by right multiplication,
$\mathcal{O}$ is the $\Gamma$-orbit of $I_{d}$ and $P$ is a nontrivial
polynomial with rational coefficients. In simple terms, it says that
this $(\mathcal{O},P)$ saturates if the Zariski closure of $\Gamma$
is big enough and $\Gamma$ has the \emph{square-free expansion hypothesis.
}This last condition means that $\Gamma$ has a finite, symmetric
generating set $\mathscr{G}$ such that the Cayley graphs\footnote{If $\mathscr{G}$ is a finite, symmetric generating set of a group
$G$, the Cayley graph of $(G,\mathscr{G})$ has vertex set $G$,
and $g_{1},g_{2}\in G$ are joined by and edge if and only if $g_{1}=sg_{2}$
for some $s\in\mathscr{G}$.} of\footnote{Recall that $\Gamma_{q}$ is the kernel of the projection $\Gamma\to\GL(d,\ZZ/q\ZZ)$.}
$(\Gamma_{q}\backslash\Gamma,\Gamma_{q}\mathscr{G})$, with $q$ running
in the positive, square-free integers, is an \emph{expander family}\footnote{Let us recall the definition of expander family of graphs. If $\mathcal{G}=(V,E)$
is a finite graph, its adjacency operator $A_{\mathcal{G}}$ acts
on $\CC$-valued functions $f$ on $V$ as follows: $(A_{\mathcal{G}}f)(v)$
is the sum of the values of $f$ at the neighbors of $v$. $A_{\mathcal{G}}$
is a self-adjoint operator on $\ell^{2}(V)$, hence it has real eigenvalues
that we denote $\lambda_{0}(\mathcal{G})\geq\lambda_{1}(\mathcal{G})\geq\cdots$.
An infinite sequence $\mathcal{G}_{n}=(V_{n},E_{n}),n\geq0$ of finite,
$d$-regular graphs is an expander family if there is an $\eps>0$
such that $\lambda_{0}(\mathcal{G}_{n})-\lambda_{1}(\mathcal{G}_{n})\geq\eps$
for any $n$, and $\#V_{n}\to\infty$ as $n\to\infty$.}. Here is the precise statement of \cite[Theorem 1.1]{bourgain-gamburd-sarnak_generalization_2011}.
\begin{thm}
\label{thm:Affine_sieve}Let $\Gamma$ be a finitely generated subgroup
of $\GL(d,\ZZ)$ whose Zariski-closure $\Gne$ is a connected, simply
connected and absolutely almost simple\footnote{These three conditions are in the sense of algebraic groups. For the
precise definitions see \cite{platonov-rapinchuk_algebraic_1994}. } subgroup of $\GLne(d)$ defined over $\QQ$. Let $P\in\QQ[\Gne]$
be a regular function on $\Gne$ that is neither zero nor a unit of
the coordinate ring $\QQ[\Gne]$, and taking integral values on $\Gamma$.
If $\Gamma$ verifies the square-free expansion hypothesis, then $R(\Gamma,P)$
is finite. Moreover, if $P$ is primitive on $\Gamma$\footnote{We say that $P$ is primitive on $\mathcal{O}$ if any integer $n\geq2$
is relatively prime to $P(x)$ for some $x\in\mathcal{O}$. This condition
ensures that the values of $P$ on $\mathcal{O}$ do not have unnecessary
common factors.}, there is an explicit upper bound for $R(\Gamma,P)$ in terms of
the spectral gap of the family of Cayley graphs of $\Gamma/\Gamma_{q}$,
where $q$ runs through the square-free natural numbers. 
\end{thm}

Theorem \ref{thm:Affine_sieve} applies to any nonelementary subgroup
$\Gamma$ of $\SL(2,\ZZ)$ with $\gexp>\frac{1}{2}$. Indeed, the
Zariski closure of any such $\Gamma$ is $\SLne(2)$ and, according
to \cite[Theorem 1.2]{bourgain-gamburd-sarnak_generalization_2011},
the square-free expansion hypothesis for $\Gamma$ is equivalent to
a uniform spectral gap for the square-free congruence coverings of
$\Gamma\backslash\Hyp$. This last condition is guaranteed by the
nonexplicit spectral gap \cite[Theorem 1.1]{bourgain-gamburd-sarnak_generalization_2011}.
To get an upper bound for the saturation number we need an explicit
spectral gap, such as Gamburd's Theorem \ref{thm:Gamburds_thm} or
our Theorem \ref{thm:MainThm}. 

Here is a concrete situation in which our main theorem can improve
known results on saturation numbers. A \emph{Pythagorean triple} is
an integral solution of the equation $x_{1}^{2}+x_{2}^{2}=x_{3}^{2}$.
The integers $|x_{1}|,|x_{2}|,|x_{3}|$ are the lengths of the sides
of a right triangle, thus the polynomials 
\[
P_{\mathtt{a}}(x)=\frac{x_{1}x_{2}}{2}\quad\text{and}\quad P_{\mathtt{h}}(x)=x_{3}
\]
give respectively the area and the length of the hypotenuse of the
triangle associated to $(x_{1},x_{2},x_{3})$ when the $x_{i}$'s
are nonnegative. Consider also the product of the coordinates $P_{\mathtt{c}}(x)=x_{1}x_{2}x_{3}$.
Any primitive Pythagorean triple $(x_{1},x_{2},x_{3})$---the greatest
common divisor of the coordinates is 1---with $x_{3}\geq0$ is of
the form $(a^{2}-b^{2},2ab,a^{2}+b^{2})$ for some $(a,b)\in\ZZ^{2}$.
Let $\Gamma$ be a finitely generated subgroup of $\SL(2,\ZZ)$ and
identify $\mathcal{O}=(1,0)\Gamma$ with its corresponding set of
Pythagorean triples. The articles \cite{kontorovich_hyperbolic_2009},
\cite{kontorovich-oh_almost_2012}, \cite{hong-kontorovich_almost_2015}
and \cite{ehrman_almost_2019} give explicit upper bounds of $R(\mathcal{O},P)$
for some\footnote{Not all the above mentioned works treat the three polynomials $P$
we are considering.} $P\in\{P_{\mathtt{a}},P_{\mathtt{c}},P_{\mathtt{h}}\}$ provided
$\gexp$ is big enough. A common feature of these works is that the
quantitative square-free expansion hypothesis for $\Gamma$ is deduced
from Gamburd's Theorem \ref{thm:Gamburds_thm}. Using instead our
Theorem \ref{thm:MainThm} when $\Gamma$ is a Schottky subgroup of
$\SL(2,\ZZ)$ with $\gexp>\frac{4}{5}$, one can extend the range
of parameters for which these results hold\footnote{If not the $R$-values, which jump, but definitely the range of $\delta$-values.}.
Similarly, our main result improves the $\delta$-range in \cite[Theorem 1.2]{bourgain-kontorovich_representations_2010}. 

\subsection{Further work: higher dimensional hyperbolic manifolds}

In a future work we will extend our Theorem \ref{thm:MainThm} to
higher dimensions. In the informal discussion that follows we give
more details and some context. 

Let $M=\Gamma\backslash\HH^{d}$ be a $d$-dimensional hyperbolic
manifold, with $\Gamma$ contained in an arithmetic subgroup of the
group $\SO^{\circ}(d,1)$ of isometries of $\HH^{d}$ preserving the
orientation. We are interested in uniform spectral gaps for congruence
coverings of $M$. As with surfaces, the most studied case is $M$
of finite volume---two examples are \cite[Théorème 1.2]{bergeron_clozel_quelques_2012}
and \cite[Theorem 1]{kelmer-silberman_uniform_2013}---. In the infinite
volume case, Magee established an explicit uniform spectral gap, akin
to Gamburd's Theorem \ref{thm:Gamburds_thm}, when $\Gamma$ is geometrically
finite, Zariski dense in $\SOne(d,1)$ and $\gexp$ is big enough---the
precise statement is \cite[Theorem 1.6]{magee_quantitative_2015}---.
We want to improve this result for $M$ convex cocompact by using
an approach similar to our proof of Theorem \ref{thm:MainThm}. This
is possible since for such an $M$, the---recurrent part of the---geodesic
flow of $M$ admits a Markov partition that yields a coding by a symbolic
dynamical system. Schottky hyperbolic surfaces are a particular instance.
Through the thermodynamic formalism we get dynamical zeta functions
of $M$, which we will use to control the eigenvalues of congruence
coverings of $M$. By a similar approach, Sarkar proved a nonexplicit
uniform spectral gap for convex cocompact $M$ in the recent work
\cite{sarkar_generalization_2022}. We aim for an explicit result
provided $\gexp$ is big enough.

\subsection{Organization of the article}

The article is divided as follows: We gather base definitions and
preliminary results in Section \ref{sec:Preliminaries}. In particular
we introduce Fuchsian Schottky groups and we explain the symbolic
dynamics of their action on $\overline{\Hyp}$. The various zeta functions
attached to a Schottky hyperbolic surface $S$---as well as the $\zeta_{\tau,n}$
used in the proof of Proposition \ref{prop:Multiplicity_bound}---are
introduced in Section \ref{sec:Zeta-functions}, where we also explain
the relation between zeros of these and the eigenvalues of $S$. The
main estimate we need to count zeros of $\zeta_{\tau,n}$ is established
in Section \ref{sec:Main-estimate}. Having introduced all the tools
we need, we prove our main result in Section \ref{sec:Main_proof}.

\subsection{Acknowledgments}

We thank Alex Kontorovich for his helpful comments on an earlier version
of the article. This project has received founding from the European
Research Council (ERC) under the European Union's Horizon 2020 research
and innovation programme (grant agreement No 949143).

\section{\label{sec:Preliminaries}Preliminaries}

This section gathers various base definitions and preliminary results.
It has four parts: In Subsection \ref{subsec:Basic_definitions} we
recall the concepts and facts about hyperbolic surfaces and its Laplace-Beltrami
operator needed in this work. Then we define Fuchsian Schottky groups
and the associated Schottky surfaces in Subsection \ref{subsec:Fuchsian-Schottky-groups}.
There, we also explain the correspondence, given a set $\Ac$ of Schottky
generators of $\Gamma$, between elements of $\Gamma$, reduced words
$w$ in $\Ac$, and certain intervals $I_{w}\subset\partial_{\infty}\Hyp$.
Subsection \ref{subsec:Dynamics-in-the-boundary} contains some lemmas
concerning the lengths of the $I_{w}$'s---like how these behave
under concatenation of words---and their relation to the word length
of $w$. Finally, in Subsection \ref{subsec:Counting-in-SL2Z} we
estimate the number of matrices of norm $\leq R$ in prime congruence
subgroups of $\SL(2,\ZZ)$.

\subsection{Basics on hyperbolic manifolds\label{subsec:Basic_definitions}}

A \emph{hyperbolic surface }is a Riemannian surface without boundary
of constant curvature $-1$. The\emph{ real hyperbolic plane} $\Hyp$
is the unique complete, simply connected hyperbolic surface. We will
work with the upper half-plane model: 

\[
\Hyp=\{z\in\CC\mid\im z>0\}.
\]
with the Riemannian metric $\frac{dx^{2}+dy^{2}}{y^{2}}$ in coordinates
$z=x+iy$. The action of $\SL(2,\RR)$ on $\HH^{2}$ by Möbius transformations
is isometric and induces an isomorphism between the group of orientation
preserving isometries of $\Hyp$ and $\PSL(2,\RR)$. We compactify
$\Hyp$ to $\overline{\Hyp}=\Hyp\cup\partial_{\infty}\HH^{2}$ by
adding the boundary $\partial_{\infty}\HH^{2}=\RR\cup\{\infty\}$.

Any hyperbolic surface $S$ we consider is assumed to be connected,
orientable and complete. Thus it can be written as $\Gamma\backslash\Hyp$
for some discrete subgroup $\Gamma$ of $\PSL(2,\RR)$ without torsion.
We say $S$ is \emph{geometrically finite }if some convex polygon
in $\Hyp$ with finitely many sides is a fundamental domain of $\Gamma$.
This is equivalent---for hyperbolic surfaces---to $\Gamma$ being
finitely generated.

Consider a discrete subgroup $\Gamma$ of $\PSL(2,\RR)$. When all
the $\Gamma$-orbits in $\overline{\Hyp}$ are infinite, we say $\Gamma$
is \emph{nonelementary}. The \emph{limit set} $\Lambda_{\Gamma}$
of $\Gamma$ is the set of accumulation points in $\overline{\Hyp}$
of $\Gamma x$ for some $x\in\Hyp$. It does not depend on the choice
of $x$. Since $\Gamma x$ is discrete, $\Lambda_{\Gamma}$ is contained
in $\partial_{\infty}\HH^{2}$. The limit set and its convex hull
$\text{Conv \ensuremath{\Lambda_{\Gamma}}}$ are closed, $\Gamma$-invariant
subsets of $\overline{\HH^{2}}$. We say that $\Gamma$ is \emph{convex
cocompact }if $\Gamma\backslash(\HH^{2}\cap\text{Conv }\Lambda_{\Gamma})$
is compact. A surface $\Gamma\backslash\Hyp$ is convex cocompact
if $\Gamma$ has this property. There are two kinds of convex cocompact
hyperbolic surfaces: in finite volume, the compact ones, and Schottky
surfaces---see Subsection \ref{subsec:Fuchsian-Schottky-groups}---in
infinite volume. The characterization of infinite volume convex cocompact
surfaces is a result of Button in \cite{button_all_1998}. The \emph{growth
exponent }$\gexp$ of $\Gamma$ is the infimum of the $s>0$ such
that 

\[
\sum_{\gamma\in\Gamma}\exp(-s\rho(x,\gamma x))<\infty,
\]
where $x\in\HH^{2}$ and $\rho$ is the hyperbolic metric. The growth
exponent is independent of the choice of $x$ and lies in the interval
$[0,1]$. For finitely generated $\Gamma$, $\gexp$ coincides with
the Hausdorff dimension of $\Lambda_{\Gamma}$. 

Let us recall the definition of the Laplace-Beltrami operator $\Delta_{X}$
of a Riemannian manifold $X$. The reader can find the details in
\cite[Chapter 4]{grigoryan_heat_2009}. Initially, $\Delta_{X}$ is
defined on the space $\mathscr{C}_{c}^{\infty}(X)$ of smooth functions
$\varphi:X\to\CC$ with compact support as follows: $\Delta_{X}\varphi$
is the divergence of the gradient of $\varphi$. As unbounded operator
of $L^{2}(X)$---with respect to the Riemannian measure of $X$---with
domain $\mathscr{C}_{c}^{\infty}(X)$, $\Delta_{X}$ is symmetric
by the Green Formula, but not self-adjoint. To remedy this one extends---using
distributional derivatives---$\Delta_{X}$ to the following closed
subspace 

\[
W_{0}^{2}(X)=\text{cl}_{W^{1}}(\mathscr{C}_{c}^{\infty}(X))\cap\{f\in W^{1}(X)\mid\Delta f\in L^{2}(X)\}
\]
of the Sobolev space $W^{1}(X)$. This extension, which we still denote
$\Delta_{X}$, is an unbounded, self-adjoint operator on $L^{2}(X)$---see
\cite[Theorem 4.6]{grigoryan_heat_2009}---with spectrum contained
in $[0,\infty)$\footnote{Taking the appropriate sign convention.}.
We will work always with this $\Delta_{X}$.

Let $S$ be a geometrically finite hyperbolic surface. Here are some
important properties of the spectrum of $\Delta_{S}$. The last two
are respectively due to Sullivan \cite{sullivan_density_1979} and
Lax-Phillips \cite{lax_asymptotic_1982}: 
\begin{enumerate}
\item When $S$ is noncompact, the continuous part of the spectrum of $\Delta_{S}$
is $\left[\frac{1}{4},\infty\right)$---see \cite[Theorem 2.12]{hislop_geometry_1994}---.
\item $\Delta_{S}$ has eigenvalues---which we sometimes call eigenvalues
of $S$---if and only if $\gexp>\frac{1}{2}$, and in that case the
smallest one is $\lambda_{0}(S)=\gexp(1-\gexp)$.
\item $S$ has finitely many eigenvalues in the interval $\left[0,\frac{1}{4}\right)$,
which we denote $\lambda_{0}(S)\leq\cdots\leq\lambda_{k}(S)$. Moreover,
these are all the eigenvalues of $\Delta_{S}$ when $S$ has infinite
volume\footnote{A finite-volume $S$ might have eigenvalues in $\left[\frac{1}{4},\infty\right)$,
even infinitely many.}. 
\end{enumerate}

\subsection{Fuchsian Schottky groups\label{subsec:Fuchsian-Schottky-groups}}

For our purposes, working with $\SL(2,\RR)$ or $\PSL(2,\RR)$ makes
no difference, so we stick to $\SL(2,\RR)$ for simplicity. By a \emph{Fuchsian
group} we mean a discrete subgroup of $\SL(2,\RR)$. We describe now
the family of Fuchsian groups relevant to our work. Consider a positive
integer $N$ and $\Ac=\{1,2,\ldots,2N\}$. For any $a\in\Ac$ we denote

\[
\mirror a=\begin{cases}
a+N & \text{if }a\leq N,\\
a-N & \text{if }a>N.
\end{cases}
\]
Suppose we are given a sequence $\mathcal{\Sc}=(D_{a},\gamma_{a})_{a\in\Ac}$
consisting of open disks $(D_{a})_{a\in\Ac}$ in $\CC$ with centers
in $\RR$ and with pairwise disjoint closures\footnote{We will denote by $\cl A$ the closure of a subset $A$ of $\CC$
instead of $\overline{A}$ to avoid confusions with the complex conjugation.}, and matrices ($\gamma_{a})_{a\in\Ac}$ in $SL(2,\RR)$ such that
$\gamma_{\mirror a}=\gamma_{a}\inv$ and

\begin{equation}
\gamma_{a}(\Hyp-D_{\mirror a})=\Hyp\cap\cl D_{a}\label{eq:def_Schottky}
\end{equation}
for any $a\in\Ac$. Such a sequence $\Sc$ will be called \emph{Schottky
data,} and we associate to it the subgroup $\Gamma_{\Sc}$ of $\SL(2,\RR)$
generated by $(\gamma_{a})_{a\in\Ac}$. A \emph{Fuchsian Schottky}
group is a subgroup $\Gamma$ of $\SL(2,\RR)$ such that $\Gamma=\Gamma_{\Sc}$
for some Schottky data $\Sc$. Here are some general properties of
a Fuchsian Schottky group $\Gamma=\Gamma_{\Sc}$: We define 
\begin{equation}
U=\cup_{a\in\Ac}D_{a}.\label{eq:def_U}
\end{equation}
The set $\Fc:=\Hyp-U$ is a fundamental domain for $\Gamma$ on $\Hyp$,
so $\Gamma$ is indeed a Fuchsian group. Moreover, $\Gamma$ is a
free group with basis $\gamma_{1},\ldots,\gamma_{N}$ by (\ref{eq:def_Schottky})
and the ping-pong argument. Note that the quotient $\Gamma\backslash\Hyp$
is a convex cocompact hyperbolic surface of infinite area since the
intersection of $\Fc$ with the convex hull of $\Lambda_{\Gamma}$
is compact, $\Fc$ has infinite area, and $\Gamma$ is torsion-free.
Conversely, the result of Button in \cite{button_all_1998} mentioned
in Subsection \ref{subsec:Basic_definitions} says that if $\Gamma_{0}\backslash\Hyp$
is a convex cocompact hyperbolic surface of infinite area, then $\Gamma_{0}$
is a Fuchsian Schottky group. 

Since $\Gamma=\Gamma_{S}$ is free, it admits a straightforward coding
by words on $\Ac$. Let us fix some terminology to describe it. A
\emph{word} with alphabet $\Ac$ is either a finite sequence $w=(a_{1},\ldots,a_{m})$
with all the $a_{j}$ in $\Ac$, or the empty word $\emptyset$. The
word $w$ is \emph{reduced} if either $w=\emptyset$, or $\mirror{a_{j}}\neq a_{j+1}$
for any $j$. We denote by $\Words$ the set of reduced words with
alphabet $\Ac$, and by $\neWords$ the set of nonempty reduced words.
Consider any $w=(a_{1},\ldots,a_{m})$ in $\neWords$. The \emph{length
}$\wlength w$ of $w$ is $m$, and we define $|\emptyset|=0$. Let
$\Words_{m}$ and $\Words_{\geq m}$ be respectively the reduced words
of length $m$ and $\geq m$. Sometimes we will write $\Start w=a_{1}$
and $\End w=a_{m}$ for the initial and last letters of $w$, respectively.
We define 
\[
w'=(a_{1},\ldots,a_{m-1})\quad\text{and}\quad\mirror w=(\mirror{a_{m}},\ldots,\mirror{a_{1}}).
\]
We say that $\mirror w$ is the \emph{mirror word} of $w$, and the
\emph{mirror set} of any $\Bc\subset\Words$ is

\[
\mirror{\Bc}=\{\mirror w\mid w\in\Bc\}.
\]
By convention $\mirror{\emptyset}=\emptyset$. Note that $\gamma_{\widetilde{w}}=\gamma_{w}\inv$.
The concatenation of two words $w_{1}=(a_{1},\ldots,a_{m}),w_{2}=(b_{1},\ldots,b_{m})$
is $w_{1}w_{2}:=(a_{1},\ldots,a_{m},b_{1},\ldots b_{m})$. By $w_{1}\to w_{2}$
we mean that $w_{1}w_{2}$ is reduced, and by $w_{1}\rightsquigarrow w_{2}$
we mean that $w_{1}$ and $w_{2}$ are nonempty and $\End{w_{1}}=\Start{w_{2}}$,
in which case $w_{1}'w_{2}$ is reduced. 

The encoding of $\Gamma=\text{\ensuremath{\Gamma_{\Sc}}}$ by $\Words$
is very transparent: We send any $w=(a_{1},\ldots,a_{m})\in\neWords$
to

\[
\gamma_{w}=\gamma_{a_{1}}\cdots\gamma_{a_{m}},
\]
and $\gamma_{\emptyset}=I$. Many properties of the surface $\Gamma\backslash\Hyp$
can be deduced by studying the action of $\Gamma$ in $\overline{\Hyp}$\footnote{The paper \cite{lalley_renewal_1989} of Lalley has several interesting
examples.}. The following notation is useful to do so: To any $w=(a_{1},\ldots,a_{m})\in\neWords$
we associate a disk in $\CC$ and an interval in $\RR$: 
\[
D_{w}:=\gamma_{w'}(D_{a_{m}}),\quad I_{w}:=D_{w}\cap\RR.
\]
We denote by $\len w$ the length of $I_{w}$. The limit set $\Lambda_{\Gamma}$
can be encoded by the set $\Words_{\infty}$ of infinite reduced words
in $\Ac^{\ZZ_{\geq1}}$ as follows: For any $w=(a_{1},a_{2},\ldots)\in\Words_{\infty}$
and any integer $n>0$, let $w_{n}=(a_{1},\ldots,a_{n})$. Note that
$w_{n}\Fc$ is contained in $D_{w_{n}}$, and that the nested sequence
of disks $D_{w_{1}}\supset D_{w_{2}}\supset\cdots$ shrinks to a point
$x_{w}\in\Lambda_{\Gamma}$. The map $\Words_{\infty}\to\Lambda_{\Gamma},w\mapsto x_{w}$
is bijective. 

\subsection{Dynamics on the boundary of $\protect\Hyp$\label{subsec:Dynamics-in-the-boundary}}

\global\long\def\consWordLengthBtauMult#1{C_{#1,10}}%
 
\global\long\def\consWordLengthBtauAdd#1{A_{#1,1}}%
 
\global\long\def\maxWL#1{L_{#1}}%
 
\global\long\def\consWordsNotLetters#1{\tau_{#1}}%

\global\long\def\consExpContractionUpper#1{\theta_{#1}}%
 
\global\long\def\consExpContractionLower#1{\eta_{#1}}%

In this subsection $\Gamma=\Gamma_{\Sc}$ is a Fuchsian Schottky group.
We collect some lemmas from \cite{bourgain_fourier_2017} and \cite{magee_explicit_2020}
that we will need about the dynamics of $\Gamma$ on $\partial_{\infty}\Hyp$.
From this point we will use the following notation for strictly positive
numbers $A,B$: $A\llG B$ means that there is a positive constant
$C_{\Gamma}$, depending only on $\Gamma$, such that $A\leq C_{\Gamma}B$.
We define $A\ggG B$ in the same fashion, and $A\asymG B$ means that
$A\llG B$ and $A\ggG B$.

We start with general considerations for Möbius transformations given
by matrices in $\SL(2,\RR)$.Consider two intervals $J_{1}=[x_{1},y_{1}],J_{2}$
contained in $\RR$. We recall a useful parametrization in \cite[Section 2.2]{bourgain_fourier_2017}
of the orientation-preserving isometries $\varphi$ of $\Hyp$ such
that $\varphi(J_{1})=J_{2}$. When $J_{1}=J_{2}=[0,1]$, a direct
computation shows that any such $\varphi$ is represented by a unique
matrix in $\SL(2,\RR)$ of the form

\begin{equation}
g_{\alpha}:=\begin{pmatrix}e^{\alpha/2} & 0\\
e^{\alpha/2}-e^{-\alpha/2} & e^{-\alpha/2}
\end{pmatrix}.\label{eq:def_g_alfa}
\end{equation}
Consider now any $J_{1},J_{2}$. For any $J=[x_{J},y_{J}]\subset\RR$
we define

\begin{equation}
T_{J}=\begin{pmatrix}|J|^{1/2} & x_{J}|J|^{-1/2}\\
0 & |J|^{-1/2}
\end{pmatrix}.\label{eq:def_TJ}
\end{equation}
Note that the Möbius transformation corresponding to $T_{J}$, that
we still denote $T_{J}$, is an affine map and $T_{J}([0,1])=J$.
Hence any $\varphi$ such that $\varphi(J_{1})=J_{2}$ is represented
by a matrix of the form

\[
T_{J_{2}}g_{\alpha}T_{J_{1}}\inv.
\]
Now suppose that $g(J_{1})=J_{2}$ for some $g\in\SL(2,\RR)$. By
the reasoning above we can write $g$ as

\[
g=RT_{J_{2}}g_{\alpha}T_{J_{1}}\inv,
\]
for some $R=\pm I$ and $\alpha\in\RR$. Let $x_{g}=g\inv(\infty)$.
Note that $\alpha$ is the unique real number such that $RT_{J_{2}}g_{\alpha}T_{J_{1}}\inv(x_{g})=\infty$.
Solving the equation we see that $\alpha$ must be

\begin{equation}
\alpha(g,J_{1})=\log\frac{g\inv(\infty)-y_{1}}{g\inv(\infty)-x_{1}}\label{eq:def_distortion_factor}
\end{equation}
when $g\inv(\infty)\neq\infty$, and $\alpha(g,J_{1})=0$ when $g\inv(\infty)=\infty$.
The number $\alpha(g,J_{1})$ is what Bourgain and Dyatlov call the
\emph{distortion factor} of $g$ on $J_{1}$. We have established
the next lemma:
\begin{lem}
\label{lem:distortion_decomposition}Let $J_{1},J_{2}\subset\RR$
be compact intervals. We can write any $g\in\SL(2,\RR)$ such that
$g(J_{1})=J_{2}$ as $g=RT_{J_{2}}g_{\alpha(g,J_{1})}T_{J_{1}}\inv$
for some $R\in\{I,-I\}$. 
\end{lem}

In the rest of this subsection we work with a Fuchsian Schottky group
$\Gamma=\Gamma_{\Sc}$. We define

\begin{equation}
\maxWL{\Gamma}=\max_{a\in\Ac}\len a.\label{eq:def_max_boundary_length}
\end{equation}

\begin{lem}
\label{lem:uniformBoundIw}Let $\Gamma$ be a Fuchsian Schottky group.
Then $|I_{w}|\leq\maxWL{\Gamma}$ for any $w\in\neWords$.
\end{lem}

\begin{proof}
Since $w$ is reduced, from (\ref{eq:def_Schottky}) we see that $D_{w}=\gamma_{w'}(D_{w})$
is contained in $D_{\Start w}$, and hence $I_{w}\subset I_{\Start w}$.
Thus $\len w\leq\maxWL{\Gamma}$.
\end{proof}
\begin{lem}
\label{lem:distortion_bound}Let $\Gamma$ be a Fuchsian Schottky
group. For any $w\in\Words_{\geq2}$ we have $|\alpha(\gamma_{w'},I_{\End w})|\llG1$.
\end{lem}

\begin{proof}
We write $w=(a_{1},\ldots,a_{n})$. Note that 
\[
\gamma_{w'}(I_{\mirror{w'}})=\gamma_{a_{1}}(I_{\mirror{a_{1}}})=\partial_{\infty}\Hyp-\cl I_{a_{1}},
\]
so $x_{w}:=\gamma_{w'}\inv(\infty)$ lies in $I_{\mirror{w'}}$. In
particular $x_{w}\in I_{\mirror{a_{n-1}}}$. Since $w$ is reduced,
$\mirror{a_{n-1}}\neq a_{n}$. The closures of the $(I_{a})_{a\in\Ac}$
are pairwise disjoint, so $|x-y|\llG1$ whenever $x\in I_{a},y\in I_{b}$
and $a\neq b$. The result follows from this observation and the formula
(\ref{eq:def_distortion_factor}) for $\alpha(\gamma_{w'},I_{a_{n}})$. 
\end{proof}
For any $w\in\neWords$ we denote $\gamma_{w'}'(z)$ by $c_{w}(z)$.
The next result is \cite[Lemma 3.2]{magee_explicit_2020}. %

\begin{lem}
\label{lem:DerVSBoundaryLength}Let $\Gamma$ be a Fuchsian Schottky
group. For any $w\in\neWords$ and any $z\in D_{\End w}$ we have

\[
|c_{w}(z)|\asymG\len w.
\]
\end{lem}

Next we restate \cite[(7), p.749]{bourgain_fourier_2017}, which implies
that the length of the intervals $I_{w}$ decreases at least exponentially
in $\wlength w$. 
\begin{lem}
\label{lem:contractionOnBoundary}For any Fuchsian Schottky group
$\Gamma$ there is a constant $\mathtt{c}_{\Gamma}\in(0,1)$ such
that

\[
\len w\leq\mathtt{c}_{\Gamma}\len{w'}
\]
 for all $w\in\Words_{\geq2}$. 
\end{lem}

The next result gives the relation between the norm of a $\gamma\in\Gamma$
and the length of its corresponding interval in $\RR$.
\begin{lem}
\label{lem:normVsBoundaryLength}Let $\Gamma$ be a Fuchsian Schottky
group. For any $w\in\neWords$ we have

\[
||\gamma_{w}||\asymG|I_{w}|^{-\frac{1}{2}}.
\]
\end{lem}

\begin{proof}
By Lemma \ref{lem:contractionOnBoundary} there is $N=N_{\Gamma}\geq2$
such that $\len w<1$ for any $w\in\Words_{\geq N}$. It suffices
to prove the result for any such $w$. First we show that $\norm{\gamma_{w}}\asymG\norm{\gamma_{w'}}$.
We write $w=(a_{1},\ldots,a_{n})$. Consider an auxiliary multiplicative
norm $\norm{\cdot}_{m}$ on the space of $M_{2}(\RR)$ of $2\times2$
real matrices. Any two norms on $M_{2}(\RR)$ are equivalent, so

\[
\norm{\gamma_{w}}\asymp\norm{\gamma_{w'}\gamma_{a_{n}}}_{m}=\norm{\gamma_{w'}}_{m}\norm{\gamma_{a_{n}}}_{m}\asymG\norm{\gamma_{w'}}.
\]
Since $\gamma_{w'}(I_{a_{n}})=I_{w}$, by Lemma \ref{lem:distortion_decomposition}
we have

\[
\gamma_{w'}=\pm IT_{I_{w}}g_{\alpha}T_{I_{a_{n}}}\inv,
\]
where $\alpha=\alpha(\gamma_{w'},I_{a_{n}})$. Using the auxiliary
norm $\norm{\cdot}_{m}$ once more we get

\[
\norm{\gamma_{w'}}\asymp\norm{T_{I_{w}}}\,\norm{g_{\alpha}}\:\norm{T_{I_{a_{n}}}\inv}.
\]
Lemma \ref{lem:distortion_bound} says that $|\alpha|\llG1$, so we
see that $\norm{g_{\alpha}}\asymG1$ from (\ref{eq:def_g_alfa}).
We also have $\norm{T_{I_{a_{n}}}\inv}\asymG1$ since 
\[
\left\{ \norm{T_{I_{a}}\inv}:a\in\Ac\right\} 
\]
is a finite set of strictly positive real numbers. Thus $\norm{\gamma_{w'}}\asymG\norm{T_{I_{w}}}$.
Finally, since $\len w<1$, from (\ref{eq:def_TJ}) we see that $\norm{T_{I_{w}}}\asymG\len w^{-1/2}$,
which completes the proof. 
\end{proof}
We will use repeatedly the fact that the length of the intervals $I_{w}$
behaves well with respect to the concatenation of words. The statement
below is a mild variation of \cite[Lemma 2.7]{bourgain_fourier_2017}.
\begin{lem}
\label{lem:concatenationEstimate}Let $\Gamma$ be a Fuchsian Schottky
group. For any $A,B\in\neWords$ such that $A\to B$ we have

\[
\len{AB}\asymp_{\Gamma}\len A\len B.
\]
\end{lem}

\begin{proof}
We use two lemmas of Bourgain and Dyatlov: \cite[Lemma 2.6]{bourgain_fourier_2017}
says that $\len{w_{1}'}\asymG\len{w_{1}}$ for any $w_{1}\in\Words_{\geq2}$.
Also, $\len{w_{1}'w_{2}}\asymG\len{w_{1}}\len{w_{2}}$ when $w_{1}\rightsquigarrow w_{2}$
by \cite[Lemma 2.7]{bourgain_fourier_2017}. Applying these to $w_{1}=A\Start B$
and $w_{2}=B$ we get
\[
\len{AB}\asymG\len{w_{1}}\len B\asymG\len A\len B,
\]
so we are done.
\end{proof}
The mirror estimate below is \cite[Lemma 2.8]{bourgain_fourier_2017}.
\begin{lem}
\label{lem:mirrorEstimate}Let $\Gamma$ be a Fuchsian Schottky group.
For any $w\in\neWords$ we have

\[
\len{\mirror w}\asymp_{\Gamma}\len w.
\]
\end{lem}

We call \emph{partition} of $\neWords$ any finite subset $\Pc$ of
$\neWords$ for which there is an integer $N>0$ such that any $w\in\Words_{\geq N}$
has a unique suffix in $\Pc$. 
\begin{lem}
\label{lem:sumLengthPartitionToDelta}Let $\Gamma$ be a Fuchsian
Schottky group. For any partition $\Pc$ of $\neWords$ we have

\[
\sum_{w\in\Pc}\len w^{\gexp}\asymG1.
\]
\end{lem}

\begin{proof}
Consider a Patterson measure $\mu$ of $\Gamma$\footnote{Nowadays we refer to these as \emph{Patterson-Sullivan measures}.
Patterson introduced them in \cite{patterson_limit_1976} for Fuchsian
groups, and Sullivan extended the construction to discrete subgroups
of isometries of the real hyperbolic space of any dimension in \cite{sullivan_density_1979}.}. Recall that $\mu$ is a probability measure on $\partial\:\Hyp$
whose support is the limit set of $\Gamma$. By \cite[Lemma 2.11, p. 755]{bourgain_fourier_2017}
we have 
\begin{equation}
\mu(\len w)\asymG\len w^{\gexp}\label{eq:ls1}
\end{equation}
 for any $w\in\neWords$. Since $\Pc$ is a partition, the intervals
$(I_{w})_{w\in\Pc}$ are pairwise disjoint and cover $\Lambda_{\Gamma}$,
hence

\begin{equation}
1=\sum_{w\in\Pc}\mu(I_{w}).\label{eq:ls2}
\end{equation}
The result follows from (\ref{eq:ls1}) and (\ref{eq:ls2}).
\end{proof}
For any $\tau>0$ we define

\begin{equation}
\Bc(\tau)=\{w\in\Words_{\geq2}\mid\len{\mirror w}\leq\tau<\len{\mirror w'}\}.\label{eq:defBtau}
\end{equation}
Note that $\Bc(\tau)$ is nonempty if and only if $\tau<\maxWL{\Gamma}$.
Although $\Bc(\tau)$ is not necessarily a partition, its mirror set
is. This fact follows easily from Lemma \ref{lem:contractionOnBoundary}.
We state it below for ease of reference.
\begin{lem}
\label{lem:BtauIsPartition}Let $\Gamma$ be a Fuchsian Schottky group.
Then $\mirror{\Bc(\tau)}$ is a partition for any $\tau\in(0,\maxWL{\Gamma})$.
\end{lem}

We also need the next part of \cite[Lemma 2.10]{bourgain_fourier_2017}:
\begin{lem}
\label{lem:boundaryLengthInBtau}Let $\Gamma$ be a Fuchsian Schottky
group and consider $\tau\in\text{(0,\ensuremath{\maxWL{\Gamma}}) }$.
Then $|I_{w}|\asymp_{\Gamma}\tau$ for any $w\in\Bc(\tau)\cup\mirror{\Bc(\tau)}$. 
\end{lem}

The next result is an estimate of the size of $\Bc(\tau)$.
\begin{lem}
\label{lem:sizeBtau}Let $\Gamma$ be a Fuchsian Schottky group and
let $\delta=\gexp$. For any $\tau\in(0,\maxWL{\Gamma})$ we have
$\#\Bc(\tau)\asymp_{\Gamma}\text{\ensuremath{\tau^{-\delta}}}$.
\end{lem}

\begin{proof}
Lemma \ref{lem:BtauIsPartition} says that $\mirror{\Bc(\tau)}$ is
a partition of $\neWords$, so

\begin{equation}
\sum_{w\in\mirror{\Bc(\tau)}}\len w^{\gexp}\asymG1\label{eq:sB1}
\end{equation}
by Lemma \ref{lem:sumLengthPartitionToDelta}. Any term in the right-hand
side of (\ref{eq:sB1}) is $\asymG\tau^{\gexp}$ according to Lemma
\ref{lem:boundaryLengthInBtau}, so $\tau^{\gexp}\#\mirror{\Bc(\tau)}\asymG1$.
The result follows since $\#\Bc(\tau)=\#\mirror{\Bc(\tau)}$.
\end{proof}
\begin{lem}
\label{lem:wordLengthBtau}Let $\Gamma$ be a Fuchsian Schottky group.
There are positive constants $\consWordLengthBtauMult{\Gamma}$ and
$\consWordLengthBtauAdd{\Gamma}$ with the next property: for any
$\tau\in(0,\maxWL{\Gamma})$ and any $w\in\Bc(\tau)$,

\begin{equation}
\consWordLengthBtauMult{\Gamma}\inv\log(\tau\inv)\leq|w|\leq\consWordLengthBtauMult{\Gamma}\log(\tau\inv)+\consWordLengthBtauAdd{\Gamma}.\label{eq:ineqWordLengthBtau}
\end{equation}
\end{lem}

\begin{proof}
We fix $w\in\Bc(\tau)$. By Lemma \ref{lem:boundaryLengthInBtau}
there is $X_{\Gamma}>1$ such that

\begin{equation}
X_{\Gamma}\inv\len w\leq\tau\leq X_{\Gamma}\len w.\label{eq:wlb1}
\end{equation}
Increasing $X_{\Gamma}$ if necessary, by Lemma \ref{lem:contractionOnBoundary}
and Lemma \ref{lem:concatenationEstimate} we assume further that

\begin{equation}
\len{w_{1}}\leq X_{\Gamma}\inv\len{w_{1}'}\label{eq:wlb2}
\end{equation}
for any $w_{1}\in\Words_{\geq2}$, and 

\begin{equation}
X_{\Gamma}\inv\len{w_{2}}\len{w_{3}}\leq\len{w_{2}w_{3}}\label{eq:wlb3}
\end{equation}
for any $w_{2},w_{3}\in\neWords$ with $w_{2}\to w_{3}$.

Using $\wlength w-1$ times (\ref{eq:wlb2}) we get $\len w\llG X_{\Gamma}^{-\wlength w}$,
which combined with the right-hand side of (\ref{eq:wlb1}) gives
$\tau\llG X_{\Gamma}^{-\wlength w}$. Taking inverse and logarithm
on both sides of this inequality proves the right-hand side of (\ref{eq:ineqWordLengthBtau}).
For the other side, let $\ell_{\Gamma}:=\min_{a\in\Ac}\len a$. Applying
$\wlength w-1$times (\ref{eq:wlb3}) we obtain $\len w\geq X_{\Gamma}^{-(\wlength w-1)}\ell_{\Gamma}^{\wlength w}$,
which combined with the left-hand side of (\ref{eq:wlb1}) yields

\[
\tau\geq\left(\frac{\ell_{\Gamma}}{X_{\Gamma}}\right)^{\wlength w}.
\]
After taking inverse and logarithm on both sides we get the left-hand
side of (\ref{eq:ineqWordLengthBtau}).
\end{proof}

\subsection{Counting in $\protect\SL(2,\protect\ZZ)_{n}$\label{subsec:Counting-in-SL2Z}}

\global\long\def\consEpsCountingSL{D_{\eps}}%

\global\long\def\div{\textbf{d}}%

In this section we estimate the number of elements of $\SL(2,\ZZ)$
congruent to $I_{2}\pmod n$ of norm $\leq R$. This will be an important
ingredient of the main lemma of Subsection \ref{subsec:Estimating-the-size}.
What we actually need in that proof is a counting in $\Gamma_{n}$,
where $\Gamma$ is a Fuchsian Schottky group contained in $\SL(2,\ZZ)$.
We do not dispose of a good method to do this directly, hence we settle
a counting in $\SL(2,\ZZ)_{n}$. In our statement we exclude the upper
and lower triangular matrices of $\SL(2,\ZZ)$. These poses no issue
for our purposes since Fuchsian Schottky groups do not have unipotents
$\neq I_{2}$. 

We endow the space of $2\times2$ real matrices with the Euclidean
norm
\[
\left|\left|\begin{pmatrix}a & b\\
c & d
\end{pmatrix}\right|\right|=\sqrt{a^{2}+b^{2}+c^{2}+d^{2}}.
\]
Let $n$ be a positive integer and let $R>0$. We define 
\[
N_{n}(R)=\left\{ \gamma=\begin{pmatrix}a & b\\
c & d
\end{pmatrix}\in\SL(2,\ZZ)_{n}\mid\norm{\gamma}\leq R,bc\neq0\right\} .
\]
 We \textbf{emphasize} the condition $bc\neq0$ above. 
\begin{lem}
\label{lem:countingSL2Z}For any $\eps>0$ there is a positive constant
$\consEpsCountingSL$ with the next property: for any integer $n\geq1$
and any $R>1$ we have

\[
N_{n}(R)\leq\consEpsCountingSL\frac{R^{\eps}}{n^{\eps}}\left(\frac{R^{2}}{n^{3}}+\frac{R}{n}+1\right).
\]
\end{lem}

We need an auxiliary result to prove Lemma \ref{lem:countingSL2Z}.
Let $\div(n)$ the number of integers dividing a nonzero integer $n$.
For example, $\div(6)=8$ since the divisors of 6 are $\pm1,\pm2,\pm3$
and $\pm6$. The next lemma follows from the formula of $\div(n)$
in terms of the decomposition of $n$ as product of prime numbers.
\begin{lem}
\label{lem:DivisorBound}For any $\eps>0$ and any nonzero integer
$n$ we have $\div(n)\ll_{\eps}|n|^{\eps}$.
\end{lem}

\begin{proof}
[Proof of Lemma \ref{lem:countingSL2Z}]Consider any 
\[
\gamma=\begin{pmatrix}a & b\\
c & d
\end{pmatrix}\in\SL(2,\ZZ)
\]
such that $\gamma\equiv I_{2}\pmod n$, $bc\neq0$ and $\norm{\gamma}\leq R$.
Since $\det\gamma=1$ and $b\neq0$, $c$ is determined once we fix
$a,b,d$. Let us bound the number of such choices. We will choose
$a,d,b$ in that order. 

Write 
\[
a=An+1,\quad b=Bn,\quad c=Cn,\quad d=Dn+1,
\]
for some integers $A,B,C,D$. Since $|a|\leq R$, there are at most
$\frac{2R}{n}+1$ choices for $a$. From $ad-bc=1$ we get $d\equiv2-a\pmod{n^{2}}$.
We also know that $|d|\leq R$, thus there are at most $\frac{2R}{n^{2}}+1$
choices for $d$.

Since $|a|\leq R$, then $|A|\leq\frac{R+1}{n}$. Similarly $|D|\leq\frac{R+1}{n}$.
Note that $\det\gamma=1$ is equivalent to 
\[
BC=AD+\frac{A+D}{n}.
\]
Given that $BC\neq0$, there are $\div\left(AD+\frac{A+D}{n}\right)$
choices for $b$. By Lemma \ref{lem:DivisorBound}

\begin{eqnarray*}
\div\left(AD+\frac{A+D}{n}\right)\ll_{\eps}\left|AD+\frac{A+D}{n}\right|^{\eps/2} & \leq & \left[\left(\frac{R+1}{n}\right)^{2}+\frac{2(R+1)}{n^{2}}\right]^{\eps/2}\\
 & \ll & \frac{R^{\eps}}{n^{\eps}}.
\end{eqnarray*}
To get the last inequality we used $R>1$. In summary,
\begin{eqnarray*}
N_{n}(R) & \ll_{\eps} & \left(\frac{2R}{n}+1\right)\left(\frac{2R}{n^{2}}+1\right)\frac{R^{\eps}}{n^{\eps}}\\
 & = & \frac{R^{\eps}}{n^{\eps}}\left(\frac{4R^{2}}{n^{3}}+\left[\frac{2}{n}+\frac{2}{n^{2}}\right]R+1\right)\\
 & \ll & \frac{R^{\eps}}{n^{\eps}}\left(\frac{R^{2}}{n^{3}}+\frac{R}{n}+1\right).
\end{eqnarray*}
\end{proof}

\section{Zeta functions and eigenvalues of hyperbolic surfaces\label{sec:Zeta-functions}}

Here we introduce various zeta functions attached to a convex cocompact
hyperbolic surface $S=\Gamma\backslash\Hyp$ of infinite volume, and
we explain the relation between eigenvalues of $S$ and the zeros
of their zeta functions. This section is has three parts: In Subsection
\ref{subsec:Selberg's-zeta-function} we define the Selberg zeta function
$Z_{S}$ of $S$ and we recall the description of its zeros and multiplicities.
To do so we define resonances of $S$, since most zeros of $Z_{S}$
are resonances. In Subsection \ref{subsec:Dynamical-zeta-functions}
we introduce the transfer operators $\Lc_{\Bc,s,\sigma}$ associated
to some Schottky data $\Sc$ of $\Gamma$, their corresponding dynamical
zeta function $\zeta_{\Bc,\sigma}$ and we give some properties of
them we will need. Finally, assuming that $\Gamma$ is contained in
$\SL(2,\ZZ)$, we introduce the zeta functions $\zeta_{\tau,n}$ we
use in the proof of Proposition \ref{prop:Multiplicity_bound} to
estimate the eigenvalues of $S_{n}$ in appropriate intervals.

\subsection{Selberg's zeta function\label{subsec:Selberg's-zeta-function}}

\global\long\def\geod#1{\mathscr{G}_{#1}}%
 
\global\long\def\Div{\text{div }}%

In this section we explain the relation between the eigenvalues of
a convex cocompact hyperbolic surface $S$ and the zeros of its Selberg
zeta function $Z_{S}$. This connection was established by Patterson
and Perry in \cite{patterson_divisor_2001}. Here we will deduce it
from a factorization of $Z_{S}$ due to Borthwick, Judge and Perry. 

Let us recall the definition of the zeta function of a connected hyperbolic
surface $X$. We denote by $\mathscr{G}_{X}$ the set of primitive
closed geodesics of $X$. The norm $N(g)$ of any $g\in\geod X$ is
defined as $e^{\ell(g)}$, where $\ell(g)$ is the length of $g$.
The \emph{Selberg zeta function} of $X$ is the holomorphic map in
the half-plane $\re s>\gexp$ defined by the infinite product

\[
Z_{X}(s)=\prod_{g\in\geod X}\prod_{k=0}^{\infty}\left(1-N(g)^{-(s+k)}\right).
\]
For a proof of the convergence see \cite[p. 33]{borthwick_spectral_2016}.
For a geometrically finite, convex cocompact hyperbolic surface $S$
there is an alternative dynamical definition of $Z_{S}$ that we recall
in Subsection \ref{subsec:Dynamical-zeta-functions}. From it one
can see that $Z_{S}$ has an entire extension, which we denote also
by $Z_{S}$. We describe below the zero set of $Z_{S}$ in terms of
divisors. Let us recall this classical definition.

The \emph{divisor} of an entire function $\varphi:\CC\to\CC$ is the
map $\CC\to\NN$ that encodes the multiplicities of the zeros of $\varphi$.
Concretely, for any $z\in\CC$, let $d_{z}:\CC\to\NN$ be the map

\[
d_{z}(w)=\begin{cases}
1 & \text{if }z=w,\\
0 & \text{if }z\neq w,
\end{cases}
\]
and let $M_{\varphi}(z)$ be the multiplicity of $z$ as zero of $\varphi$.
The divisor of $\varphi$ is

\[
\Div\varphi=\sum_{z\in\CC}M_{\varphi}(z)d_{z}.
\]
To give explicitly the divisor of $Z_{S}$ we need to introduce a
spectral data of $S$ finer than the eigenvalues. 

The \emph{resolvent} of $S$ is the meromorphic family of bounded
operators of $L{{}^2}(S)$ given by

\[
R_{S}(s)=(\Delta_{S}-s(1-s)I)\inv.
\]
It is initially defined on the half-plane $\re s>\frac{1}{2}$, where
the poles are the such that $\lambda=s(1-s)$ is an eigenvalue of
$S$. By the work \cite{mazzeo_meromorphic_1987} of Mazzeo and Melrose
we know how to reduce the domain of $\Delta_{S}$ to extend $R_{S}$
to all $\CC$ as a meromorphic family of operators $\mathscr{C}_{c}^{\infty}(S)\to\mathscr{C}^{\infty}(S)$%
, that we still denote $R_{S}$. A \emph{resonance} of $S$ is a pole
of $R_{S}$. We denote by $\mathcal{R}_{S}$ the set of resonances
of $S$. Following \cite[Definition 8.2, pp. 145]{borthwick_spectral_2016},
the multiplicity $m_{s}$ of a resonance $s$ of $S$ is defined as

\[
m_{s}=\text{rank }\left(\int_{C}R_{S}(s)ds\right),
\]
where $C$ is an anti clockwise oriented circle enclosing $s$ and
no other resonance of $S$. The multiplicity of a resonance $s$ with
$\re s>\frac{1}{2}$ coincides with the multiplicity of the eigenvalue
$\lambda=s(1-s)$. The next expression of $\Div Z_{S}$ follows from
a factorization of $Z_{S}$ that is a particular case of \cite[Theorem 10.1, pp. 213]{borthwick_spectral_2016}
\footnote{This is a factorization of the $Z_{X}$ of any nonelementary, geometrically
finite hyperbolic surface $X$ of infinite area.}. 
\begin{prop}
Let $S$ be a nonelementary, geometrically finite, convex cocompact
hyperbolic surface. Then 
\[
\Div Z_{S}=-\chi(S)\sum_{n=0}^{\infty}(2n+1)d_{-n}+\sum_{s\in\mathcal{R}_{S}}m_{s}d_{s}.
\]
\end{prop}

For ease of reference we record in the next corollary the correspondence
between resonances of $S$ and zeros of $Z_{S}$. We denote by $\ZZ_{\leq0}$
the set of integers $\leq0$.
\begin{cor}
\label{cor:relResonancesZerosSZetaFunction}Let $S$ be a nonelementary,
geometrically finite, convex cocompact hyperbolic surface. The resonances
of $S$ and the zeros of $Z_{S}$ coincide in $\CC-\ZZ_{\leq0}$.
Moreover, this correspondence respects multiplicities.
\end{cor}

\subsection{Dynamical zeta functions\label{subsec:Dynamical-zeta-functions}}

Here we introduce the dynamical zeta functions $\zeta_{\Bc,\sigma}$
associated to a Fuchsian Schottky group $\Gamma=\Gamma_{\Sc}$. The
parameters $\Bc$ and $\sigma$ are respectively a finite subset of
$\Words_{\geq2}$ and a finite-dimensional unitary representation
of $\Gamma$. The intuition is that we choose $\Bc$ depending on
the scale at which we want study the dynamics of $\Gamma$ in $\Lambda_{\Gamma}$,
whereas $\sigma$ serves to represent a dynamical zeta function of
a finite covering of $S$ as a zeta function on $S$. We call them
dynamical zeta functions because they are defined in terms of transfer
operators $\Lc_{\Bc,s,\sigma}$ coming from the action of $\Gamma$
on the open subset $U$\footnote{This was defined in (\ref{eq:def_U}).}
of $\CC$. We start by introducing the operators $\Lc_{\Bc,s,\sigma}$
and stating their main properties. Then, we define the dynamical zeta
functions $\zeta_{\Bc,\sigma}$ and we explain how to use them to
count eigenvalues of coverings of $S$.

Let $\Gamma=\Gamma_{\Sc}$ be a Fuchsian Schottky group. First we
recall the definition of the classical transfer operator of $\Gamma$.
Remember that 
\[
U=\bigcup_{a\in\Ac}D_{a}.
\]
Consider some complex number $s$ and a measurable map $f:U\to\CC$.
We define $\Lc_{s}f:U\to\CC$ on $D_{b}$ as

\begin{equation}
\Lc_{s}f(z)=\sum_{\substack{a\in\Ac\\
a\neq\mirror b
}
}\gamma_{a}'(z)^{s}f(\gamma_{a}(z)).\label{eq:def_classical_transf_op}
\end{equation}
The term $\gamma_{a}'(z)^{s}$ needs some explanation. Since $\gamma_{a}'$
does not take negative values on $U-D_{\mirror a}$\footnote{Let $\gamma$ be any Möbius transformation of $\widehat{\CC}$. A
direct computation shows that if $\im z\neq0$, then $\gamma'(z)<0$
if and only if $\re z=\gamma\inv(\infty)$. When $\gamma=\gamma_{a}$
we know that $\gamma_{a}\inv(\infty)$ is in $D_{\mirror a}$ because
$\gamma_{a}(D_{\mirror a})=\widehat{\CC}-\cl(D_{a})$.} 

, we define $\gamma_{a}'(z)^{s}$ as $\exp(sF(\gamma_{a}'(z))$, where
$F$ is the branch of the logarithm $\CC-(-\infty,0]\to\{\rho\in\CC\mid|\im\rho|<\pi\}$.
The map $\Lc_{s}$ is part of a family of transfer operators%
{} of $\Gamma$ that we define now. Consider a finite subset $\Bc$
of $\Words_{\geq2}$, a complex number $s$ and a unitary representation
$(\sigma,V)$ of $\Gamma$. For any measurable map $f:U\to V$ we
define $\Lc_{\Bc,s,\sigma}f:U\to V$ on $D_{b}$ as

\begin{equation}
\Lc_{\Bc,s,\sigma}f(z)=\sum_{\substack{w\in\Bc\\
w\rightsquigarrow b
}
}c_{w}(z)^{s}\sigma(\gamma_{w'}\inv)f(\gamma_{w'}(z)).\label{eq:def_general_transf_op}
\end{equation}
Recall that $c_{w}$ means $\gamma'_{w'}$. Again, the terms $c_{w}(z)$
appearing in (\ref{eq:def_general_transf_op})are not in $(-\infty,0]$,
thus we define their $s$-th power as we did in (\ref{eq:def_classical_transf_op}).
Note that $\Lc_{s}=\Lc_{\Bc,s,\sigma}$ for $\Bc=\Words_{2}$ and
$\sigma$ the trivial one-dimensional representation of $\Gamma$. 

Let us fix the functional spaces where the transfer operators act
for our purposes. For any Hilbert space $V$, the space

\[
\Hc(U,V)=\left\{ f:U\to V\mid f\text{ is holomorphic,\ensuremath{\int_{U}\norm{f(z)}dz<\infty}}\right\} 
\]
endowed with the inner product 
\[
\scal fg=\int_{U}\scal{f(u)}{g(u)}_{V}dz
\]
is a Hilbert space. The integrals in $U$ are taken with respect to
the Lebesgue measure. $\Hc(U,V)$ is a \emph{Bergman space} of $U$.
The classical Bergman space of $U$ is $\Hc(U,\CC)$, and we denote
it $\Hc(U)$. Note that for any unitary representation $(\sigma,V_{\sigma})$
of $\Gamma$, $\Lc_{\Bc,s,\sigma}$ preserves $\Hc(U,V_{\sigma})$
. From now on we consider $\Lc_{\Bc,s,\sigma}$ as an operator on
its respective Bergman space. To define the dynamical zeta functions
of $\Gamma$ we need the following fact, which is \cite[Lemma 4.1]{magee_explicit_2020}. 
\begin{lem}
\label{lem:TransOpAreTraceClass}Let $\Gamma$ be a Fuchsian Schottky
group. For any nonempty finite subset $\Bc$ of $\Words_{\geq2}$,
any $s\in\CC$ and any finite dimensional unitary representation $\sigma$
of $\Gamma$, the transfer operator $\Lc_{\Bc,s,\sigma}$ is trace
class. 
\end{lem}

Since any $\Lc_{\Bc,s,\sigma}$ as above is trace class, then $I-\Lc_{\Bc,s,\sigma}$
has a well-defined Fredholm determinant. The zeta function $\zeta_{\Bc,\sigma}$
of $\Gamma$ is the map

\[
\zeta_{\Bc,\sigma}(s)=\det(I-\Lc_{\Bc,s,\sigma}).
\]
It is entire because $s\mapsto\Lc_{\Bc,s,\sigma}$ is an analytic
family of bounded operators of $\Hc(U,V_{\sigma})$ . 

One of the main reason for us to care about the family of functions
$\zeta_{\Bc,\sigma}$ is that it encompasses the Selberg zeta functions
of all the finite coverings of $S=\Gamma\backslash\Hyp$. This is
part of \cite[Proposition 4.4]{magee_explicit_2020}, and it implies
that $Z_{S}$ has an entire extension.

\begin{thm}
\label{thm:DynamicalZetaEqualSelbergZeta}Let $\Gamma'$ be a finite-index
subgroup of a Fuchsian Schottky group $\Gamma$ and let $\sigma$
be the regular representation of $\Gamma$ in $\ell^{2}(\Gamma'\backslash\Gamma)$.
Then $\zeta_{\Words_{2},\sigma}$ coincides with the Selberg zeta
function of $\Gamma'\backslash\Hyp$ in the domain of the last one.
\end{thm}

Before proceeding our discussion of the functions $\zeta_{\Bc,\sigma}$
we state a property of the transfer operators that we will need in
Section \ref{sec:Main-estimate}. Since $\Lc_{\Bc,s,\sigma}$ is trace
class, it is in particular a Hilbert-Schmidt operator. Here is an
explicit formula of the Hilbert-Schmidt norm of $\Lc_{\Bc,s,\sigma}$
taken from \cite[Lemma 4.7, p. 156]{magee_explicit_2020}. We use
the following notation: Let $\Bc$ be a subset of $\Words$ and consider
$a,b\in\Ac$. Then $\Bc_{a}^{b}$ is the set of $w\in\Bc$ whose first
and last letters are respectively $a$ and $b$. For any $a\in\Ac$
we denote $B_{a}$ the Bergman kernel of $D_{a}$\footnote{The Bergman kernel of a disk $D$ in $\CC$ of radius $r$ and center
$c$ is the map $B_{D}:D\times D\to\CC$ given by 
\begin{equation}
B_{D}(w,z)=\frac{r^{2}}{\pi\left[r^{2}-(w-c)\overline{(z-c)}\right]^{2}}.\label{eq:BergmanKFromula}
\end{equation}
It is characterized by the fact that $f(w)=\int_{D}B_{D}(w,z)f(z)dz$
for any $f\in\Hc(D)$ and any $w\in D$. See \cite[Chapter 1]{duren_bergman_2014}
for a proof. }.
\begin{lem}
\label{lem:HS_norm_formula}Let $\Gamma$ be a Fuchsian Schottky group.
For any finite subset $\Bc$ of $\Words_{\geq2}$, any $s\in\CC$,
and any finite dimensional unitary representation $\sigma$ of $\Gamma$,
the Hilbert-Schmidt norm of $\Lc_{\Bc,s,\sigma}$ is given by
\end{lem}

\[
||\Lc_{\Bc,s,\sigma}||_{HS}^{2}=\sum_{a,b\in\Ac}\sum_{w_{1},w_{2}\in\Bc_{a}^{b}}\tr\sigma(\gamma{}_{w_{2}}\gamma{}_{w_{1}}\inv)\int_{D_{b}}c_{w_{1}}(z)^{s}\overline{c_{w_{2}}(z)^{s}}B_{a}(\gamma_{w'_{1}}(z),\gamma_{w'_{2}}(z))dz.
\]

The last two lemmas will allow us to count eigenvalues of a finite
covering of $\Gamma\backslash\Hyp$ by counting zeros of a suitable
$\zeta_{\Bc,\sigma}$. The first follows from basic properties of
Fredholm determinants. The second and its proof are similar to \cite[Lemma 2.4]{dyatlov-zworski_fractal_2020}.
\begin{lem}
\label{lem:characterizationZerosZetaBSigma}Let $\Gamma$ be a Fuchsian
Schottky group Consider a finite subset $\Bc$ of $\Words_{\geq2}$
and a finite dimensional unitary representation $\sigma$ of $\Gamma$.
Then $s\in\CC$ is a zero of $\zeta_{\Bc,\sigma}$ if and only if
1 is an eigenvalue of $\Lc_{\Bc,s,\sigma}$. Moreover, the multiplicity
of the zero $s$ is the dimension of $\ker(I-\Lc_{\Bc,s,\sigma})$.
\end{lem}

\begin{lem}
\label{lem:fixedPointsTransfOps}Let $\Gamma=\Gamma_{\Sc}$ be a Fuchsian
Schottky group. Consider a partition $\Zc$ of $\Words$ and a finite
dimensional unitary representation $\sigma$ of $\Gamma$. Any function
$f\in\Hc(U,V_{\sigma})$ fixed by $\Lc_{\Words_{2},s,\sigma}$ is
also fixed by $\Lc_{\mirror{\Zc},s,\sigma}$.
\end{lem}

\subsection{The zeta functions of the main proof\label{subsec:The-zeta-functions}}

\global\long\def\consPointwiseBoundZetaFunc#1{c_{#1}}%

\global\long\def\tauGamma{\tau_{\Gamma}}%

In this subsection, $\Gamma=\Gamma_{\Sc}$ is a Fuchsian Schottky
subgroup of $\SL(2,\ZZ)$. We introduce the zeta functions $\zeta_{\tau,n}$
that we will use in the proof of Theorem \ref{thm:MainThm} to count
eigenvalues of congruence coverings of $\Gamma\backslash\Hyp$. There
are three statements here. The first one is the relation between eigenvalues
of $\Gamma_{n}\backslash\Hyp$ and zeros of $\zeta_{\tau,n}$. The
other two will be used to estimate the zeros of $\zeta_{\tau,n}$
in an appropriate disk. 

We recall some notation: $\Gamma_{n}$ is the $n$-th principal congruence
subgroup of $\Gamma$, $L_{\Gamma}$ is the supremum of the $\len w$
over all $w\in\neWords$, and $\Bc(\tau)$ is the subset of $\Words_{\geq2}$
defined in (\ref{eq:defBtau}), which is finite and nonempty for any
$\tau\in(0,\maxWL{\Gamma})$. Let $\sigma_{n}$ be the regular representation
of $\Gamma$ on $\ell^{2}(\Gamma_{n}\backslash\Gamma)$. For any $\tau\in(0,\text{\ensuremath{\maxWL{\Gamma}}) }$
and any positive integer $n$ we define

\begin{equation}
\zeta_{\tau,n}(s)=\det(1-\Lc_{\Bc(\tau),s,\text{\ensuremath{\sigma_{n}}}}^{2}).\label{eq:defZetaTauN}
\end{equation}
To see that this definition makes sense, note that $\Lc_{\Bc(\tau),s,\sigma_{n}}^{2}$
is trace class since $\Lc_{\Bc(\tau),s,\sigma_{n}}$is trace class
by Lemma \ref{lem:TransOpAreTraceClass}, and the product of two trace
class operators is trace class. Hence $1-\Lc_{\Bc(\tau),s,\sigma_{n}}^{2}$
has a well-defined Fredholm determinant.

Here is the relation between eigenvalues of $\Gamma_{n}\backslash\Hyp$
and zeros of $\zeta_{\tau,\sigma_{n}}$.
\begin{prop}
\label{prop:eigenvaluesAreZerosRefZetaFunc}Let $\Gamma$ be a Fuchsian
Schottky subgroup of $\SL(2,\ZZ)$. For any $\tau\in(0,\maxWL{\Gamma})$
and any positive integer $n$, if $s_{0}(1-s_{0})$ is an eigenvalue
of $\Gamma_{n}\backslash\Hyp$ of multiplicity $m_{0}$ and $s_{0}>\frac{1}{2}$,
then $s_{0}$ is a zero of $\zeta_{\tau,n}$ of multiplicity at least
$m_{0}$.
\end{prop}

\begin{proof}
We denote $\dim_{\CC}\ker(I-\Lc_{\Bc(\tau),s,\sigma_{n}}^{2})$ by
$m_{1}$. By Lemma \ref{lem:characterizationZerosZetaBSigma} the
result we want is equivalent to $m_{1}\geq m_{0}$. Note that $m_{0}$
is the dimension of $\ker(I-\Lc_{\Words_{2},s,\sigma_{n}})$ by Corollary
\ref{cor:relResonancesZerosSZetaFunction}, Theorem \ref{thm:DynamicalZetaEqualSelbergZeta}
and Lemma \ref{lem:characterizationZerosZetaBSigma}. We know that
$\mirror{\Bc(\tau)}$ is a partition by Lemma \ref{lem:BtauIsPartition},
thus $m_{0}\leq\dim_{\CC}\ker(I-\Lc_{\Bc(\tau),s,\sigma_{n}})$. It
is immediate that $\dim_{\CC}\ker(I-\Lc_{\Bc(\tau),s,\sigma_{n}})\leq m_{1}$,
so we are done. 
\end{proof}
We defined $\zeta_{\tau,n}$ using the square of $\Lc_{\Bc(\tau),s,\sigma_{n}}$
in order to have the next upper bound of $|\zeta_{\tau,n}|$ in terms
of the Hilbert-Schmidt norm of $\Lc_{\Bc(\tau),s,\sigma_{n}}$, rather
than the trace norm. The advantage of the Hilbert-Schmidt norm is
that we have the explicit formula of Lemma \ref{lem:HS_norm_formula}
for it. 
\begin{lem}
\label{lem:logZetaVsHSnorm}Let $\Gamma$ be a Fuchsian Schottky subgroup
of $\SL(2,\ZZ)$. For any positive integer $n$, any $\tau\in(0,\maxWL{\Gamma})$
and any $s\in\CC$ we have

\[
\log|\zeta_{\tau,n}(s)|\leq\norm{\Lc_{\Bc(\tau),s,\sigma_{n}}}_{HS}^{2}.
\]
\end{lem}

\begin{proof}
Since $\Lc_{\Bc(\tau),s,\sigma_{n}}^{2}$ is trace class, by \cite[Theorem A.32, p.439]{borthwick_spectral_2016}

\[
\log|\zeta_{\tau,n}(s)|\leq\trNorm{\Lc_{\Bc(\tau),s,\sigma_{n}}^{2}}.
\]
To conclude note that $\trNorm{\Lc_{\Bc(\tau),s,\sigma_{n}}^{2}}\leq\HSnorm{\Lc_{\Bc(\tau),s,\sigma_{n}}}^{2}$
by the Cauchy-Schwartz inequality on the space of Hilbert-Schmidt
operators on $\mathcal{H}(U,\ell^{2}(\Gamma/\Gamma_{n}))$---see
\cite[p. 439]{borthwick_spectral_2016}---.
\end{proof}
We will estimate the number of zeros of $\zeta_{\tau,n}$ in a disk
with center $c\in[0,\infty)$ using Jensen's Formula. It turns out
that the contribution of $|\zeta_{\tau,n}(c)|$ is negligible for
any big enough $c$. The next statement is a special case of \cite[Proposition 4.8]{magee_explicit_2020}.
\begin{prop}
\label{prop:PointwiseBoundZetaFunc}Let $\Gamma$ be Fuchsian Schottky
subgroup of $\SL(2,\ZZ)$. There are real numbers $\text{\ensuremath{\consPointwiseBoundZetaFunc{\Gamma}}}>0$
and $\tauGamma\in(0,1)$ with the next property: for any $c\geq\consPointwiseBoundZetaFunc{\Gamma}$,
any $\tau\in(0,\tauGamma)$ and any integer $n>0$ we have

\[
-\log|\zeta_{\tau,n}(c)|\leq\tau[\Gamma:\Gamma_{n}].
\]
\end{prop}

\section{\label{sec:Main-estimate}The main estimate}

We will estimate the number of zeros of the zeta functions $\zeta_{\tau,n}$
of a Schottky subgroup of $\SL(2,\ZZ)$ using Jensen's Formula, which
we state now for convenience. The proof can be found most complex
analysis books, like \cite[Chapter 5]{stein_complex_2003}.
\begin{thm}
\label{thm:JensensFormula}Let $\zeta$ be an holomorphic map on a
neighborhood of a closed disk $D=\{z\in\CC\mid|z-c|\leq R\}$, and
let $z_{1},\ldots,z_{m}$ be the zeros of $\zeta$ in $D$ repeated
according to multiplicity. If $0<|z_{j}-c|<R$ for each $j$, then

\[
\sum_{j=1}^{m}\log\frac{R}{|z_{j}-c|}=\frac{1}{2\pi}\int_{0}^{2\pi}\log|\zeta(Re^{i\theta}+c)|d\theta-\log|\zeta(c)|.
\]
\end{thm}

We already have Proposition \ref{prop:PointwiseBoundZetaFunc} to
deal with the term $-\log|\zeta_{\tau,n}(c)|$. The purpose of this
section is to establish a pointwise estimate for $|\zeta_{\tau,n}|$
in the half-plane $\re s>0$, which we will use to bound the integral
appearing in Jensen's Formula. Thanks to Lemma \ref{lem:logZetaVsHSnorm}
we can achieve our goal by bounding the Hilbert-Schmidt norm of $\Lc_{\Bc(\tau),s,\sigma_{n}}$.

\subsection{The bound of the Hilbert-Schmidt norm}

\global\long\def\consHSbound{C_{\Gamma,1}}%

Consider a Fuchsian Schottky group $\Gamma$ contained in $\SL(2,\ZZ)$,
a positive integer $n$ and $\tau\in\text{(0,1) }$. We define

\[
\Pc_{n}(\tau)=\{(w_{1},w_{2})\in\Bc(\tau)^{2}\mid\Start{w_{1}}=\Start{w_{2}},\End{w_{1}}=\End{w_{2}},\gamma_{w_{1}}\gamma\inv_{w_{2}}\in\Gamma_{n}\}.
\]
Recall that $\maxWL{\Gamma}$ is the supremum of $\len w$ over all
$w\in\neWords$. The next bound for the Hilbert-Schmidt norm of $\Lc_{\Bc(\tau),s,\sigma_{n}}$
follows from the explicit formula of Lemma \ref{lem:HS_norm_formula}.
\begin{lem}
\label{lem:HS_bound}Let $\Gamma$ be a Fuchsian Schottky subgroup
of $\SL(2,\ZZ)$. There is a constant $\consHSbound>0$ with the next
property: For any $\tau\in(0,\maxWL{\Gamma})$, any $s\in\CC$ with
$\re s>0$, and any integer $n>0$ we have

\[
||\Lc_{\Bc(\tau),s,\sigma_{n}}||_{HS}^{2}\ll_{\text{\ensuremath{\Gamma}}}[\Gamma:\Gamma_{n}]e^{2\pi|\im s|}\consHSbound^{\re s}\tau^{2\re s}\#\Pc_{n}(\tau).
\]
\end{lem}

\begin{proof}
Consider the basis of $\ell^{2}(\Gamma_{n}\backslash\Gamma)$ formed
by the maps $e_{x}:\Gamma_{n}\backslash\Gamma\to\CC$, 
\[
e_{x}(y)=\begin{cases}
1 & \text{if }x=y,\\
0 & \text{otherwise,}
\end{cases}
\]
for any $x\in\Gamma_{n}\backslash\Gamma$. Note that for any $\gamma\in\Gamma$,
$\sigma_{n}(\gamma)$ permutes the $e_{x}$. More specifically, $\sigma_{n}(\gamma)e_{x}=e_{x\gamma\inv}$,
so $\tr\sigma_{n}(\gamma)$ is the number of fixed points of the right
multiplication $m_{\gamma\inv}$ on $\Gamma_{n}\backslash\Gamma$
by $\gamma\inv$. Since $\Gamma_{n}$ is a normal subgroup of $\Gamma$,
$m_{\gamma\inv}$ is the identity when $\gamma$ is in $\Gamma_{n}$,
and has no fixed points otherwise. Thus the formula for the Hilbert-Schmidt
norm of Lemma \ref{lem:HS_norm_formula} applied to $\Lc_{\Bc(\tau),s,\sigma_{n}}$
reduces to 

\begin{equation}
\HSnorm{\Lc_{\Bc(\tau),s,\sigma_{n}}}^{2}=[\Gamma:\Gamma_{n}]\sum_{(w_{1},w_{2})\in\Pc_{n}(\tau)}\int_{D_{\End{w_{1}}}}c_{w_{1}}(z)^{s}\overline{c_{w_{2}}(z)^{s}}B_{\Start{w_{1}}}(\gamma_{w_{1}'}(z),\gamma_{w_{2}'}(z))dz.\label{eq:HS0}
\end{equation}
Recall that we defined $c_{w}(z)^{s}$ using $\log:\CC-(-\infty,0]\to\{\omega\in\CC\mid|\im\omega|<\pi\}$,
so

\[
|c_{w_{j}}(z)^{s}|=e^{-\theta_{j}(z)\im s}|c_{w_{j}}(z)|^{\re s},
\]
where $\theta_{j}(z)$ is an argument for $c_{w_{j}}(z)$ in $(-\pi,\pi)$.
Note that $|c_{w_{j}}(z)|\asymG\len{w_{j}}$ by Lemma \ref{lem:DerVSBoundaryLength},
and since $w_{j}$ is in $\Bc(\tau)$, then $\len{w_{j}}\asymG\tau$
by Lemma \ref{lem:boundaryLengthInBtau}. Thus 
\[
|c_{w_{j}}(z)|\leq\consHSbound^{1/2}\tau,
\]
 for some $\consHSbound>0$. Since $\re s>0$, then

\begin{equation}
|c_{w_{j}}(z)^{s}|\leq e^{\pi|\im s|}\consHSbound^{\re s/2}\tau^{\re s}.\label{eq:HS1}
\end{equation}
Now we bound the term $B_{\Start{w_{1}}}(\gamma_{w_{1}'}(z),\gamma_{w_{2}'}(z))$.
Note that $\gamma_{w_{j}'}(z)$ is the disk $D_{w_{j}}$. Since the
closures of the disks $D_{a},a\in\Ac$ are pairwise disjoint and $|w_{j}|\geq2$,
the distance from $\cl D_{w_{j}}$ to the boundary of $D_{\Start{w_{1}}}$
has a positive lower bound $\eta$ that depends only on $\Gamma$.
Then we see that 
\begin{equation}
|B_{\Start{w_{1}}}(\gamma_{w_{1}'}(z),\gamma_{w_{2}'}(z))|\llG1\label{eq:HS2}
\end{equation}
using the formula (\ref{eq:BergmanKFromula}) of the Bergman kernel
of a disk. For any $(w_{1},w_{2})\in\Pc_{n}(\tau)$, (\ref{eq:HS1})
and (\ref{eq:HS2}) yield

\begin{eqnarray}
\int_{D_{\End{w_{1}}}}\left|c_{w_{1}}(z)^{s}\overline{c_{w_{2}}(z)^{s}}B_{\Start{w_{1}}}(\gamma_{w_{1}'}(z),\gamma_{w_{2}'}(z))\right|dz & \llG & \text{\ensuremath{e^{2\pi|\text{\ensuremath{\im}s| }}\consHSbound^{\re s}\tau^{2\re s}area(D_{\End{w_{1}}})}}\nonumber \\
 & \llG & e^{2\pi|\text{\ensuremath{\im}s| }}\consHSbound^{\re s}\tau^{2\re s}\label{eq:HS3}
\end{eqnarray}
To obtain the estimate of the statement we apply the triangle inequality
to the right-hand side of (\ref{eq:HS0}) and then use (\ref{eq:HS3}).
\end{proof}

\subsection{Estimating the size of $\protect\Pc_{n}(\tau)$\label{subsec:Estimating-the-size}}

\global\long\def\consPtauACnonempty{C_{\Gamma,2}}%
 
\global\long\def\consSizePtauAC{F_{\eps}}%
 
\global\long\def\consMaxWLength{G_{\Gamma}}%

\global\long\def\sump{\Scal_{p}^{*}}%
\global\long\def\sumn{\Scal_{n}^{*}}%

Here $\Gamma$ will also denote a Fuchsian Schottky subgroup of $\SL(2,\ZZ)$.
To bound $|\zeta_{\tau,n}|$ on the half-plane $\re s>0$ using the
estimate for $\HSnorm{\Lc_{\Bc(\tau),s,\sigma_{n}}}$ in Lemma \ref{lem:HS_bound},
we need to estimate $\#\Pc_{n}(\tau)$ from above. That is the goal
of this subsection. 

The bound is given in the next lemma, which we prove at the end of
this section. We use the following notation: 

\[
\mathtt{K}_{\tau}:=\log(\tau\inv)+1,\quad\mathtt{L}_{\tau,n}=\log\left(\frac{\tau\inv}{n}\right)+1.
\]
For any $\eps>0$, let $\consEpsCountingSL$ be as in Lemma \ref{lem:countingSL2Z}. 
\begin{lem}
\label{lem:sizePtau}Let $\Gamma$ be a Fuchsian Schottky subgroup
of $\SL(2,\ZZ)$. For any positive integer $n$, any $\tau\in(0,1)$
less than $\maxWL{\text{\ensuremath{\Gamma}}}$, and any $\eps>0$
we have

\begin{equation}
\#\Pc_{n}(\tau)\llG\consEpsCountingSL\mathtt{K}_{\tau}^{3}\mathtt{L}_{\tau,n}\left(\left[\frac{\tau^{-(2+\eps)}}{n^{3+\eps}}+\frac{\tau^{-(1+\eps)}}{n^{1+\eps}}\right]\mathtt{L}_{\tau,n}+\frac{\tau^{-\gexp}}{n^{\gexp}}\right)+\tau^{-\gexp}.\label{eq:sizePtau}
\end{equation}
\end{lem}

The next easy corollary will simplify the computations for the upper
bound of $m_{\Gamma}(n,\beta)$.

\begin{cor}
Let $\Gamma$ be a Fuchsian Schottky subgroup of $\SL(2,\ZZ)$. For
any integer $n\geq2$, any $\tau$ in $\left(0,\min\left\{ \frac{1}{n+1},\maxWL{\text{\ensuremath{\Gamma}}}\right\} \right)$,
and any $\eps>0$ we have
\begin{equation}
\#\Pc_{n}(\tau)\llG\consEpsCountingSL(\log(\tau\inv))^{5}\left(\frac{\tau^{-(2+\eps)}}{n^{3+\eps}}+\frac{\tau^{-(1+\eps)}}{n^{1+\eps}}\right)+\tau^{-\gexp}.\label{eq:sizePtauSimplified}
\end{equation}
\end{cor}

\begin{proof}
It suffices to show that the right-hand side of (\ref{eq:sizePtau}),
which we denote $M(\eps,\tau,n)$, is less or equal---up to multiplication
by a constant---than the right-hand side of (\ref{eq:sizePtauSimplified})
for any $\tau$ and $n$ as in the statement. Since $X:=\frac{\tau\inv}{n}>\frac{n+1}{n}>1$,
then $X^{\gexp}<X^{1+\eps}$ and $\mathtt{L}_{\tau,n}>1$. Note also
that $\mathtt{L}_{\tau,n}<\mathtt{K_{\tau}}$, so

\begin{eqnarray}
M(\eps,\tau,n) & = & \consEpsCountingSL\mathtt{K}_{\tau}^{3}\mathtt{L}_{\tau,n}\left(\left[\frac{X^{(2+\eps)}}{n}+X^{1+\eps}\right]\mathtt{L}_{\tau,n}+X^{\gexp}\right)+\tau^{-\gexp}\nonumber \\
 & \ll & \consEpsCountingSL\mathtt{K}_{\tau}^{5}\left(\frac{X^{(2+\eps)}}{n}+X^{1+\eps}\right)+\tau^{-\gexp}.\label{eq:CC1}
\end{eqnarray}
Finally, $\tau\inv\geq n+1\geq3$, so $\mathtt{K}_{\tau}<2\log(\tau\inv)$,
which combined with (\ref{eq:CC1}) gives the result.
\end{proof}
We already know by Lemma \ref{lem:sizeBtau} that the contribution
to $\#\Pc_{n}(\tau)$ of the diagonal of $\Bc(\tau)^{2}$ is roughly
$\tau^{-\gexp}$. Let us focus on the complement

\[
\Pc_{n}^{*}(\tau):=\{(w_{1},w_{2})\in\Pc_{n}(\tau)\mid w_{1}\neq w_{2}\}.
\]
To estimate the size of $\Pc_{n}^{*}(\tau)$ we partition it into
subsets which are easier to count. These depend on two parameters
$a,c\in\NN$. We use the following notation: For any pair $(w_{1},w_{2})\in\Pc_{n}^{*}(\tau)$
there is a unique decomposition 

\[
w_{1}=AB_{1}C,\quad w_{2}=AB_{2}C
\]
such that $A,C\in\neWords,B_{1},B_{2}\text{\ensuremath{\in\Words}},A\to B_{i}\to C$,
$\Start{B_{1}}\neq\Start{B_{2}}$ and $\End{B_{1}}\neq\End{B_{2}}$.
Since the $w_{i}$ are in $\Bc(\tau),$ Lemma \ref{lem:concatenationEstimate}
and Lemma \ref{lem:boundaryLengthInBtau} imply that

\begin{equation}
\len A\cdot\len{B_{i}}\cdot\len C\asymG\tau.\label{eq:IAIBIC}
\end{equation}
This observation will be used in the sequel. We set 

\[
\consMaxWLength=2\maxWL{\Gamma}.
\]
For any nonnegative integers $a,c$ we define

\[
\Pc_{n}^{*}(\tau,a,c)=\left\{ (w_{1},w_{2})\in\Pc_{n}^{*}(\tau)\mid2^{-(a+1)}\leq\frac{\wlength{I_{A}}}{\consMaxWLength}<2^{-a},2^{-(c+1)}\leq\frac{\wlength{I_{C}}}{\consMaxWLength}<2^{-c}\right\} .
\]
Here is an upper bound of $\#\Pc_{n}^{*}(\tau,a,c)$. Recall that
$D_{\eps}$ is as in Lemma \ref{lem:countingSL2Z}.
\begin{lem}
\label{lem:size_PtauAC}Let $\Gamma$ be Fuchsian Schottky subgroup
of $\SL(2,\ZZ)$. Consider an integer $n\geq2$ and $\tau\in(0,\min\{1,\maxWL{\Gamma}\})$.
For any $a,c\in\NN$ and any $\eps>0$ we have

\begin{equation}
\Pc_{n}^{*}(\tau,a,c)\ll_{\Gamma}\consEpsCountingSL\mathtt{K_{\tau}^{3}}\left(\frac{\tau^{-(2+\eps)}}{n^{3+\eps}}+\frac{\tau^{-(1+\eps)}}{n^{1+\eps}}+2^{(a+c)(\gexp-\eps)}\frac{\tau^{-\eps}}{n^{\eps}}\right).\label{eq:size_PtauAC}
\end{equation}
\end{lem}

To establish Lemma \ref{lem:size_PtauAC} we compare suitable upper
and lower bounds of

\begin{equation}
\sumn(\tau,a,c):=\sum_{(w_{1},w_{2})\in\Pc_{n}^{*}(\tau,a,c)}\len A^{\gexp}\len{B_{2}\text{\ensuremath{\mirror{B_{2}}}}}^{\gexp/2}\len C^{\gexp}.\label{eq:defSump}
\end{equation}
We start with the easiest one.
\begin{lem}
\label{lem:sumpVsPtauAC}Let $\Gamma$ be a Fuchsian Schottky subgroup
of $\SL(2,\ZZ)$. For any $\tau\in(0,\maxWL{\Gamma})$, any nonnegative
integers $a,c$, and any integer $n\geq1$ we have

\[
\sumn(\tau,a,c)\asymG\tau^{\gexp}\#\Pc_{n}^{*}(\tau,a,c).
\]
\end{lem}

\begin{proof}
By Lemma \ref{lem:concatenationEstimate} and Lemma \ref{lem:mirrorEstimate}
we have $\len{B_{1}\mirror{B_{2}}}\asymG\len{B_{i}}\cdot\len{B_{2}}$.
Thus each term on the right-hand side of (\ref{eq:defSump}) is $\asymG\tau^{\gexp}$
by (\ref{eq:IAIBIC}), and the result follows.
\end{proof}
We pass to the upper bound of $\sumn(\tau,a,c)$.
\begin{lem}
\label{lem:sumpUppBound}Let $\Gamma$ be a Fuchsian Schottky subgroup
of $\SL(2,\ZZ)$. For any integer $n\geq2$, any $\eps>0$, any $\tau\in(0,1)$and
any $a,c\in\mathbb{N}$ we have

\[
\mathscr{S}_{n}^{*}(\tau,a,c)\llG\consEpsCountingSL\tau^{\gexp}\mathtt{K_{\tau}^{3}}\left(\frac{\tau^{-(2+\eps)}}{n^{3+\eps}}+\frac{\tau^{-(1+\eps)}}{n^{1+\eps}}+2^{(a+c)(\gexp-\eps)}\frac{\tau^{-\eps}}{n^{\eps}}\right).
\]
\end{lem}

\begin{proof}
For any $(w_{1},w_{2})\in\Pn(\tau,a,c)$, the word lengths of $A,B_{1},B_{2}$
and $C$ are at most $\max\{|w_{1}|,|w_{2}|\}$, which by Lemma \ref{lem:wordLengthBtau}
is at most 

\[
E_{\Gamma}(\tau):=\max\left\{ \consWordLengthBtauMult{\Gamma},\consWordLengthBtauAdd{\Gamma}\right\} (\log(\tau\inv)+1).
\]
Let $\Ec_{n}(\tau,a,c)$ and $\Mc_{n}(\tau,a,c)$ be respectively
the set of $w\in\Words$ such that $w\in\{A,C\}$ and $w=B_{1}\mirror{B_{2}}$,
for some $(w_{1},w_{2})\in\Pn(\tau,a,c)$. Let $M_{n}(\tau,a,c)$
be the maximum of the $\len B$, for $B\in\Mc_{n}(\tau,a,c)$. Any
$B\in\Mc_{n}(\tau,a,c)$ has word length $\leq2E_{\Gamma}(\tau)$,
thus the map

\[
\Pn(\tau,a,c,)\to\Ec_{n}(\tau,a,c)^{2}\times\Mc_{n}(\tau,a,c),(w_{1},w_{2})\mapsto(A,C,B_{1}\mirror{B_{2}})
\]
is at most $F_{\Gamma}(\tau)$-to-one, where $F_{\Gamma}(\tau):=2E_{\Gamma}(\tau)+1$.
Hence, 

\begin{eqnarray}
\sumn(\tau,a,c) & \leq & F_{\Gamma}(\tau)\sum_{A\in\Ec_{n}(\tau,a,c)}\sum_{C\in\Ec_{n}(\tau,a,c)}\sum_{B\in\Mc_{n}(\tau,a,c)}\len A^{\gexp}\len C^{\gexp}\len B^{\gexp/2}\nonumber \\
 & = & F_{\Gamma}(\tau)\left(\sum_{w\in\Ec_{n}(\tau,a,c)}\len w^{\gexp}\right)^{2}\sum_{B\in\Mc_{n}(\tau,a,c)}\len B^{\gexp/2}\nonumber \\
 & \llG & \mathtt{K_{\tau}}\left(\sum_{w\in\mathcal{W}_{\leq E_{\Gamma}(\tau)}}\len w^{\gexp}\right)^{2}M_{n}(\tau,a,c)^{\gexp/2}\#\Mc_{n}(\tau,a,c).\label{eq:sump1}
\end{eqnarray}
We will bound separately the last three factors of (\ref{eq:sump1}).
Applying Lemma \ref{lem:sumLengthPartitionToDelta} to each of the
partitions $\Words_{N}$ of $\Words$, for $N\leq E_{\Gamma}(\tau)$,
we deduce

\begin{equation}
\sum_{w\in\Words_{\leq E_{\Gamma}(\tau)}}\len w^{\gexp}\asymG E_{\Gamma}(\tau)\asymG K_{\tau}.\label{eq:sumpB1}
\end{equation}
Let us handle $M_{n}(\tau,a,c)$ and $\#\Mc_{n}(\tau,a,c)$. Any $B=B_{1}\mirror{B_{2}}\in\Mc_{n}(\tau,a,c)$
is associated to some $(w_{1},w_{2})\in\Pn(\tau,a,c)$. We know that
$\len B\asymG\len{B_{1}}\cdot\len{B_{2}}$ by lemmas \ref{lem:concatenationEstimate}
and \ref{lem:mirrorEstimate}, so (\ref{eq:IAIBIC}) yields

\begin{equation}
\wlength{I_{B}}\asymG\len A^{-2}\len C^{-2}\tau^{2}\asymG2^{2(a+c)}\tau^{2}.\label{eq:Mp_low_bound}
\end{equation}
In particular

\begin{equation}
M_{n}(\tau,a,c)^{\gexp/2}\llG2^{(a+c)\gexp}\tau^{\gexp}.\label{eq:sumpB2}
\end{equation}
To bound the size of $\Mc_{n}(\tau,a,c)$ we use a counting argument
in $\SL(2,\ZZ)$. Let 
\[
\gamma_{B}=\begin{pmatrix}a & b\\
c & d
\end{pmatrix}.
\]
Note that $\gamma_{B}$ is in $\SL(2,\ZZ)_{n}-\{I_{2}\}$ by definition
of $\Pn(\tau)$, $bc\neq0$---since $\Gamma$ is Schottky, its only
unipotent is $I_{2}$ and $\Gamma$ has no torsion---and that $\norm{\gamma_{B}}\llG2^{-(a+c)}\tau\inv$
by Lemma \ref{lem:normVsBoundaryLength} and (\ref{eq:Mp_low_bound}).
To simplify the notation we define $X=\frac{2^{-(a+c)}\tau\inv}{n}$.
Thus for any $\eps>0$ Lemma \ref{lem:countingSL2Z} yields
\begin{equation}
\#\Mc_{n}(\tau,a,c)\llG\consEpsCountingSL\left(\frac{X^{2+\eps}}{n}+X^{1+\eps}+X^{\eps}\right).\label{eq:sumpB3}
\end{equation}

To conclude we plug (\ref{eq:sumpB1}), (\ref{eq:sumpB2}) and (\ref{eq:sumpB3})
in (\ref{eq:sump1}):

\begin{eqnarray*}
\mathscr{S}_{n}^{*}(\tau,a,c) & \llG & \mathtt{K}_{\tau}^{3}2^{(a+c)\gexp}\tau^{\gexp}\consEpsCountingSL\left(\frac{X^{2+\eps}}{n}+X^{1+\eps}+X^{\eps}\right)\\
 & = & \consEpsCountingSL\tau^{\gexp}\mathtt{K}_{\tau}^{3}\left(2^{(a+c)(\gexp-2-\eps)}\frac{\tau^{-(2+\eps)}}{n^{3+\eps}}+\cdots\right.\\
 &  & \quad\cdots\left.+2^{(a+c)(\gexp-1-\eps)}\frac{\tau^{-(1+\eps)}}{n^{1+\eps}}+2^{(a+c)(\gexp-\eps)}\frac{\tau^{-\eps}}{n^{\eps}}\right)\\
 & \leq & \consEpsCountingSL\tau^{\gexp}\mathtt{K}_{\tau}^{3}\left(\frac{\tau^{-(2+\eps)}}{n^{3+\eps}}+\frac{\tau^{-(1+\eps)}}{n^{1+\eps}}+2^{(a+c)(\gexp-\eps)}\frac{\tau^{-\eps}}{n^{\eps}}\right).
\end{eqnarray*}
\end{proof}
We can now estimate $\#\Pc_{n}^{*}(\tau,a,c)$. 
\begin{proof}
[Proof of Lemma \ref{lem:size_PtauAC}] The result follows from the
inequality between the lower and upper bounds of $\mathscr{S}_{n}^{*}(\tau,a,c)$
given respectively by Lemma \ref{lem:sumpVsPtauAC} and Lemma \ref{lem:sumpUppBound}.
\end{proof}
To get the upper bound of $\#\Pc_{n}(\tau)$ we need to know for which
pairs $(a,c)\in\NN^{2}$ the set $\Pc_{n}^{*}(\tau,a,c)$ is nonempty.
The next lemma gives a necessary condition.

\begin{lem}
\label{lem:Cond_a_c}Let $\Gamma$ be Fuchsian Schottky subgroup of
$\SL(2,\ZZ)$. There is a constant $\mathtt{H}_{\Gamma}>0$ with the
next property: If $\mathcal{P}_{n}^{*}(\tau,a,c)$ is nonempty for
some natural numbers $a,c,n$ with $n\geq2$, and $\tau\in(0,1)$,
then

\begin{equation}
2^{a+c}\leq\mathtt{H}_{\Gamma}\frac{\tau\inv}{n}.\label{eq:cond_a_c}
\end{equation}
\end{lem}

\begin{proof}
Consider $(w_{1},w_{2})\in\Pc_{n}^{*}(\tau,a,c)$, the corresponding
$A,B_{1},B_{2},C$ and $B:=B_{1}\mirror{B_{2}}$. To establish the
inequality we will compare a lower and an upper bound of $\norm{\gamma_{B}}$.
Since $w_{1}\neq w_{2}$, then $\gamma_{B}\neq I_{2}$. We also know
that $\gamma_{B}\equiv I_{2}\pmod n$, so
\begin{equation}
\norm{\gamma_{B}}\geq\norm{\begin{pmatrix}1-n & 0\\
0 & 1
\end{pmatrix}}\gg n.\label{eq:QQ1}
\end{equation}
From (\ref{eq:IAIBIC}) and the definition of $\Pc_{n}^{*}(\tau,a,c)$
follows that

\[
\wlength{I_{B_{1}}}\asymG\frac{\tau}{\wlength{I_{A}}\cdot\wlength{I_{C}}}\asymG2^{a+c}\tau,
\]
and similarly for $\wlength{I_{B_{2}}}$. Applying Lemma \ref{lem:normVsBoundaryLength},
Lemma \ref{lem:concatenationEstimate} and Lemma \ref{lem:mirrorEstimate}
we get

\begin{equation}
\norm{\gamma_{B}}\asymG\wlength{I_{B}}^{-\frac{1}{2}}\asymG(\wlength{I_{B_{1}}}\cdot\wlength{I_{B_{2}}})^{-\frac{1}{2}}\asymG2^{-a-c}\tau\inv.\label{eq:QQ2}
\end{equation}
The result follows from (\ref{eq:QQ1}) and (\ref{eq:QQ2}).
\end{proof}
We prove now the main result of the section.
\begin{proof}
[Proof of Lemma \ref{lem:sizePtau}]Let $\Bc^{\Delta}(\tau)$ be the
diagonal of $\Bc(\tau)^{2}$. Then 
\[
\#\Pc_{n}(\tau)=\#\Pc_{n}^{*}(\tau)+\#\Bc^{\Delta}(\tau).
\]
Since $\#\Bc^{\Delta}(\tau)=\#\Bc(\tau)$ and $\tau\leq\maxWL{\Gamma}$,
Lemma \ref{lem:sizeBtau} gives 
\begin{equation}
\#\Bc^{\Delta}(\tau)\asymG\tau^{-\gexp}.\label{eq:sizeDiagonalBtau}
\end{equation}

Now we prove two intermediate inequalities to estimate $\#\Pc_{n}^{*}(\tau)$.
Recall the notation 
\[
\mathtt{K}_{\tau}:=\log(\tau\inv)+1,\quad\mathtt{L}_{\tau,n}:=\log\left(\frac{\tau\inv}{n}\right)+1.
\]
Consider $\mathtt{H}_{\Gamma}$ as in Lemma \ref{lem:Cond_a_c}, the
greatest integer $N$ such that 
\[
2^{N}\leq\mathtt{H}_{\Gamma}\frac{\tau\inv}{n},
\]
and 
\[
\mathcal{I}_{n}(\tau)=\{(a,c)\in\NN\mid a+c\leq N\}.
\]
Then $\Pc_{n}^{*}(\tau)$ is the disjoint union of the $\Pc_{n}^{*}(\tau,a,c)$
with $(a,c)\in\mathcal{I}_{n}(\tau)$ by Lemma \ref{eq:cond_a_c}.
If $(a,c)$ is in $\mathcal{I}_{n}(\tau)$, then $a$ and $c$ are
$\llG\mathtt{L}_{\tau,n}$. Thus

\begin{equation}
\#\mathcal{I}_{n}(\tau)\llG\mathtt{L}_{\tau,n}^{2}.\label{eq:size_I_n}
\end{equation}
Let us denote $\frac{\tau\inv}{n}$ by $X$. We also have

\begin{eqnarray}
\sum_{(a,c)\in\mathcal{I}_{n}(\tau)}2^{(\gexp-\eps)(a+c)} & = & \sum_{m=0}^{N}(m+1)2^{(\gexp-\eps)m}\nonumber \\
 & = & \frac{(N+1)2^{(\gexp-\eps)(N+2)}-(N+2)2^{(\gexp-\eps)(N+1)}+1}{(2^{\gexp-\eps}-1)^{2}}\nonumber \\
 & \llG & N2^{(\gexp-\eps)N}\nonumber \\
 & \llG & \mathtt{L}_{\tau,n}X^{\gexp-\eps}.\label{eq:sum_I_n}
\end{eqnarray}

We estimate the size of $\Pc_{n}^{*}(\tau)$ using Lemma \ref{lem:size_PtauAC},
(\ref{eq:size_I_n}) and (\ref{eq:sum_I_n}):

\begin{eqnarray}
\#\Pc_{n}^{*}(\tau) & = & \sum_{(a,c)\in\mathcal{I}_{n}(\tau)}\#\Pc_{n}^{*}(\tau,a,c)\nonumber \\
 & \llG & \consEpsCountingSL\mathtt{K}_{\tau}^{3}\left(\left[\frac{X^{2+\eps}}{n}+X^{1+\eps}\right]\#\mathcal{I}_{n}(\tau)+X^{\eps}\sum_{(a,c)\in\mathcal{I}_{n}(\tau)}2^{(a+c)(\gexp-\eps)}\right)\\
 & \llG & \consEpsCountingSL\mathtt{K}_{\tau}^{3}\mathtt{L}_{\tau,n}\left(\left[\frac{X^{2+\eps}}{n}+X^{1+\eps}\right]\mathtt{L}_{\tau,n}+X^{\gexp}\right).\label{eq:sizePtauStar}
\end{eqnarray}

The upper bound for $\#\Pc(\tau)$ follows from (\ref{eq:sizeDiagonalBtau})
and (\ref{eq:sizePtauStar}).
\end{proof}

\section{\label{sec:Main_proof}Proof of the main result}

Here we complete the proof of Theorem \ref{thm:MainThm}. Let $S=\Gamma\backslash\Hyp$
be a Schottky surface with $\Gamma$ contained in $\SL(2,\ZZ)$. Recall
that $J_{\beta}=[0,\beta(1-\beta)]$ and, for any integer $n\geq1$,
$m_{\Gamma}(n,J_{\beta})$ is the number of eigenvalues of $S_{n}=\Gamma_{n}\backslash\Hyp$
in $J_{\beta}$. This section is divided into three parts: \ref{subsec:upper_bound}
is devoted to the proof of Proposition \ref{prop:Multiplicity_bound},
an upper bound of $m_{\Gamma}(n,J_{\beta})$ when $\gexp>\frac{4}{5}$
and $\beta$ lies in $(t_{\Gamma},\gexp)$. Then we give in \ref{subsec:lower_bound}
a lower bound for the multiplicity of new eigenvalues of $S_{n}$
provided that all the prime factors of $n$ are big enough. Finally,
in \ref{subsec:main_proof} we combine these bounds to establish our
main result.

\subsection{The upper bound of $m_{\Gamma}(n,J_{\beta})$\label{subsec:upper_bound}}

\global\long\def\consChoiceAlphaMult#1#2{\mathtt{Z}(#1,#2)}%

We isolate the next computation from the proof of Proposition \ref{prop:Multiplicity_bound}
for the sake of clarity. Let $D_{\eps}$ be as in Lemma \ref{lem:countingSL2Z}
and, for any positive numbers $\delta$ and $\beta$, let 
\[
t_{\delta}:=\frac{\delta}{6}+\frac{2}{3},\quad\text{and}\quad\ell(\delta,\beta):=t_{\delta}+\frac{\beta-t_{\delta}}{4}.
\]

\begin{lem}
\label{lem:ChoiceAlpha}For any $\delta\in\left(\frac{4}{5},1\right]$
and any $\beta\in(t_{\delta},\delta]$ there are $\mathtt{Z}=\consChoiceAlphaMult{\delta}{\beta}>0,$
$\eps=\eps(\delta,\beta)\in(0,1)$ and $\alpha=\alpha(\delta,\beta)\in(2,5)$
with the next property: For any $\beta_{1}\in[\ell(\delta,\beta),\beta]$
and any $x>1$ we have 

\begin{equation}
x^{-2\alpha\beta_{1}}\left[D_{\eps}(\alpha\log x)^{5}\left(x^{\alpha(2+\eps)-3-\eps}+x^{(\alpha-1)(\eps+1)}\right)+x^{\text{\ensuremath{\alpha\delta}}}\right]\leq\mathtt{Z}x^{-2-\eps}.\label{eq:ChoiceAlpha}
\end{equation}
\end{lem}

\begin{proof}
We fix $\delta,\beta,\beta_{1}$ and $x$ as in the statement. Now
consider any $\eps_{1}\in\left(0,\frac{3}{5}\right)$ and $\alpha_{1}\in(0,5)$.
Let $M(\alpha_{1},\eps_{1})$ be the left-hand side of (\ref{eq:ChoiceAlpha})
replacing $\alpha$ and $\eps$ respectively by $\alpha_{1}$ and
$\eps_{1}$. From the fact that $x>1$ and $\alpha_{1}<5$ we easily
get 

\[
(\alpha_{1}\log x)^{5}\text{\ensuremath{\ll_{\eps_{1}}x^{\eps_{1}}}},
\]
and thus

\begin{eqnarray}
M(\alpha_{1},\eps_{1}) & \ll_{\eps_{1}} & x^{-2\alpha_{1}\beta_{1}}\left[x^{\eps_{1}}\left(x^{\alpha_{1}(2+\eps_{1})-3-\eps_{1}}+x^{(\alpha_{1}-1)(\eps_{1}+1)}\right)+x^{\text{\ensuremath{\alpha_{1}\delta}}}\right]\nonumber \\
 & \text{=} & x^{\alpha_{1}(2+\eps_{1}-2\beta_{1})-3}+x^{-\alpha_{1}(2\beta_{1}-1-\eps_{1})-1}+x^{-\alpha_{1}(2\beta_{1}-\delta)}.\label{eq:X1}
\end{eqnarray}
We will show that the three exponents on the right-hand side of (\ref{eq:X1})
are $\leq-2-\eps_{1}$ if $\eps_{1}$ and $\alpha_{1}$ are well chosen.
Using that $\frac{4}{5}<t_{\delta}<\frac{5}{6},0<\eps_{1}<\frac{3}{5}$
and $\delta\leq1$ we readily see that 
\[
2+\eps_{1}-2\beta_{1},\quad2\beta_{1}-1-\eps_{1},\quad\text{and}\quad2\beta_{1}-\delta
\]
are strictly positive, so the exponents in (\ref{eq:X1}) are $\leq-2-\eps_{1}$
if and only if the next inequalities hold: 

\begin{eqnarray}
\alpha_{1} & \leq & \frac{1-\eps_{1}}{2+\eps_{1}-2\beta_{1}},\label{eq:X2.1}\\
\alpha_{1} & \geq & \frac{1+\eps_{1}}{2\beta_{1}-1-\eps_{1}},\label{eq:X2.2}\\
\alpha_{1} & \geq & \frac{2+\eps_{1}}{2\beta_{1}-\delta}.\label{eq:X2.3}
\end{eqnarray}
Take now $\eps_{1}\leq\frac{1}{15}$. Then 
\[
\frac{2+\eps_{1}}{2\beta_{1}-\delta}>\frac{1+\eps_{1}}{2\beta_{1}-1-\eps_{1}},
\]
and hence the system reduces to 

\[
\frac{2+\eps_{1}}{2\beta_{1}-\delta}\leq\alpha_{1}\leq\frac{1-\eps_{1}}{2+\eps_{1}-2\beta_{1}}.
\]
Since $\beta_{1}\geq\ell(\delta,\beta)$, then

\[
\frac{2+\eps_{1}}{2\beta_{1}-\delta}\leq\frac{2+\eps_{1}}{2\ell(\delta,\beta)-\delta},\quad\text{and}\quad\frac{1-\eps_{1}}{2+\eps_{1}-2\ell(\delta,\beta)}\leq\frac{1-\eps_{1}}{2+\eps_{1}-2\beta_{1}}.
\]
Reducing $\eps_{1}$ if necessary, one has
\begin{equation}
\frac{2+\eps_{1}}{2\ell(\delta,\beta)-\delta}\leq\frac{1-\eps_{1}}{2+\eps_{1}-2\ell(\delta,\beta)}.\label{eq:X3}
\end{equation}
Indeed, note that the denominators on (\ref{eq:X3}) are positive.
Multiplying by their product and simplifying we see that (\ref{eq:X3})
is equivalent to 
\[
\eps_{1}^{2}+(4-\delta)\eps_{1}\leq\frac{3}{2}(\beta-t_{\delta}).
\]

Let $r(\delta,\beta)$ be the positive root of $X^{2}+(4-\delta)X-\frac{3}{2}(\beta-t_{\delta})$.
The above reasoning shows that for 
\[
\eps=\min\left\{ r(\delta,\beta),\frac{1}{15}\right\} \quad\text{and}\quad\alpha=\frac{2+\eps}{2\ell(\delta,\beta)-\delta},
\]
one has $M(\alpha,\eps)\ll_{\eps}x^{-2-\eps}$. This can be reformulated
as (\ref{eq:ChoiceAlpha}) since $\eps$ depends only on $\delta$
and $\beta$. Clearly $\eps$ belongs to $(0,1)$. As for $\alpha$,
from the assumptions on $\delta$ and $\beta$ we easily get that
$t_{\delta}\in\left(\frac{4}{5},1\right)$ and $\ell(\delta,\beta)\in\left(\frac{4}{5},\delta\right)$.
A direct computation shows then that $\alpha$ is in the interval
$(2,5)$.
\end{proof}
We are ready to establish the upper bound of $m_{\Gamma}(n,J_{\beta})$.
\begin{proof}
[Proof of Proposition \ref{prop:Multiplicity_bound}]We define 
\[
\consMultiplicityp=\max\{\tauGamma,\maxWL{\Gamma}\}^{-\frac{1}{2}},
\]
where $\tauGamma$ and $\maxWL{\Gamma}$ are respectively as in Proposition
\ref{prop:PointwiseBoundZetaFunc} and (\ref{eq:def_max_boundary_length}).
Consider any $n>\consMultiplicityp$. 

The projection $\Gamma\to\SL(2,\ZZ/n\ZZ)$ identifies $\Gamma/\Gamma_{n}$
with a subgroup of $\SL(2,\ZZ/n\ZZ)$, so\footnote{For any prime $p$ and any integer $a\geq1$ we have---see \cite[eq. (7.1)]{bourgain-gamburd_expansion_2008}---$\#\SL(2,\ZZ/p^{a}\ZZ)=p^{3a-2}(p^{2}-1)<p^{3a}$.
For any integer $n\geq2$, its decomposition $p_{1}^{a_{1}}\cdots p_{k}^{a_{k}}$
as product of primes give an isomorphism between $\SL(2,\ZZ/n\ZZ)$
and the product of the $\SL(2,\ZZ/p_{j}^{a_{j}}\ZZ)$. It follows
that $\#\SL(2,\ZZ/n\ZZ)<n^{3}$.}

\begin{equation}
[\Gamma:\Gamma_{n}]\leq\#\SL(2,\ZZ/n\ZZ)<n^{3}.\label{eq:M-1}
\end{equation}
To lighten the notation we write $m(n,\beta)$ instead of $m_{\Gamma}(n,J_{\beta})$
for the number of eigenvalues of $\Gamma_{n}\backslash\Hyp$ in $J_{\beta}=[0,\beta(1-\beta)]$
counted with multiplicity. We label these as 

\[
\lambda_{n,1}\leq\cdots\leq\lambda_{n,m(n,\beta).}
\]
Consider $s_{n,j}\in\left(\frac{1}{2},\gexp\right]$ such that $\lambda_{n,j}=s_{n,j}(1-s_{n,j})$.

We fix $\eps=\eps(\gexp,\beta)\in(0,1)$ and $\alpha=\alpha(\gexp,\beta)\in(2,5)$
as in Lemma \ref{lem:ChoiceAlpha}. Let $\tau_{n}=n^{-\alpha}$. The
$s_{n,j}$ are zeros of $\zeta_{\tau_{n},n}$ by Proposition \ref{prop:eigenvaluesAreZerosRefZetaFunc}
since $\tau_{n}<n^{-2}<\maxWL{\Gamma}$. Hence $m(n,\beta)$ is less
or equal than the number of zeros of $\zeta_{\tau_{n},n}$ inside
any circle $\mathscr{C}$ containing $[\beta,\gexp]$, which we'll
bound with Jensen's Formula---for a convenient $\mathscr{C}$---.

We have $\tau_{n}<\tauGamma$, so by Proposition \ref{prop:PointwiseBoundZetaFunc}
there is $\consPointwiseBoundZetaFunc{\Gamma}$ such that

\begin{equation}
-\log|\zeta_{\tau_{n},n}(c')|\leq[\Gamma:\Gamma_{n}]\tau_{n},\label{eq:M0}
\end{equation}
for any $c'\geq\consPointwiseBoundZetaFunc{\Gamma}$. Recall that
$t_{\Gamma}=\frac{\gexp}{6}+\frac{2}{3}$ and $\ell(\gexp,\beta)=t_{\Gamma}+\frac{\beta-t_{\Gamma}}{4}$.
The zeta function $\zeta_{\tau_{n},n}$ is holomorphic and not identically
zero, so it has countably many zeros. Thus we can pick $\beta_{1}\in\left[\ell(\gexp,\beta),\frac{t_{\Gamma}+\beta}{2}\right)$
and $c\in[\consPointwiseBoundZetaFunc{\Gamma}+1,\consPointwiseBoundZetaFunc{\Gamma}+2]$
such that $\zeta_{\tau_{n},n}$ has no zeros in $\mathscr{C}\cup\{c\}$,
where $\mathscr{C}$ is the circle of radius $R:=c-\beta_{1}$ and
center $c$. Since the $s_{n,j}$'s lie in the interval $[\beta,c)$
we have 

\begin{equation}
m(n,\beta)\log\left(\frac{R}{c-\beta}\right)\leq\sum_{j=1}^{m(n,\beta)}\log\left(\frac{R}{c-s_{n,j}}\right).\label{eq:M1}
\end{equation}
Note that\footnote{Indeed, $\frac{R}{c-\beta}=1+\frac{\beta-\beta_{1}}{c-\beta}\geq1+\frac{\beta-t_{\Gamma}}{2(\consPointwiseBoundZetaFunc{\Gamma}+2)}$,
so we can take $A(\Gamma,\beta)=\log\inv\left(1+\frac{\beta-t_{\Gamma}}{2(\consPointwiseBoundZetaFunc{\Gamma}+2)}\right)$.} $\log\inv\left(\frac{R}{c-\beta}\right)\leq A(\Gamma,\beta)$ for
some positive constant $A(\Gamma,\beta)$. Thus, from (\ref{eq:M1})
and Jensen's Formula---Theorem \ref{thm:JensensFormula}---applied
to $\zeta_{\tau_{n},n}$ and $\mathscr{C}$ we obtain

\begin{equation}
m(n,\beta)\leq A(\Gamma,\beta)\left(\frac{1}{2\pi}\int_{0}^{2\pi}\log|\zeta_{\tau_{n},n}(Re^{i\theta}+c)|d\theta-\log|\zeta_{\tau_{n},n}(c)|\right).\label{eq:M1.5}
\end{equation}
Let us bound the right-hand side of (\ref{eq:M1.5}). From (\ref{eq:M-1})
and (\ref{eq:M0}) we obtain

\begin{equation}
-\log|\zeta_{\tau_{n},n}(c)|<n^{3-\alpha}.\label{eq:M2}
\end{equation}
For any $s\in\CC$ Lemma \ref{lem:logZetaVsHSnorm} gives 

\[
\log|\zeta_{\tau_{n},n}(s)|\leq\HSnorm{\Lc_{\Bc(\tau_{n}),s,\sigma_{n}}}^{2},
\]
since $\tau_{n}<\maxWL{\Gamma}$. Assume now that $s$ is in $\mathscr{C}$.
Using the upper bound for the Hilbert-Schmidt norm of $\Lc_{\Bc(\tau_{n}),s,\sigma_{n}}$,
the estimation of the size of $\Pc_{n}(\tau_{n})$---respectively
Lemma \ref{lem:HS_bound} and Corollary \ref{eq:sizePtauSimplified}---,
that $|\im s|\llG1$ and $\re s\geq\beta_{1}$, we get

\[
\log|\zeta_{\tau_{n},n}(s)|\llG n^{3}n^{-2\alpha\beta_{1}}\left[D_{\eps}(\alpha\log n)^{5}(n^{\alpha(2+\eps)-3-\eps}+n^{(\alpha-1)(1+\eps)})+n^{\alpha\gexp}\right].
\]
By our choice of $\eps=\eps(\gexp,\beta)$ and $\alpha=\alpha(\gexp,\beta)$,
Lemma \ref{lem:ChoiceAlpha} implies that

\[
\log|\zeta_{\tau_{n},n}(s)|\leq\mathtt{E}_{\Gamma,\beta}n^{1-\eps},
\]
for some $\mathtt{E}_{\Gamma,\beta}>0$. Integrating over $\mathscr{C}$
we get

\begin{equation}
\frac{1}{2\pi}\int_{0}^{2\pi}\log|\zeta_{\tau_{n},n}(Re^{i\theta}+c)|d\theta\leq\mathtt{E}_{\Gamma,\beta}n^{1-\eps}.\label{eq:M3}
\end{equation}

We are ready to conclude. Note that $\xi:=\max\{1-\eps,3-\alpha\}$
lies in $(0,1)$ since $\eps\in(0,1)$ and $\alpha>2$. Using (\ref{eq:M2})
and (\ref{eq:M3}) in (\ref{eq:M1.5}) we see that 

\begin{equation}
m(n,\beta)\leq\consMultiplicityCoeff{\beta}n^{\xi},\label{eq:M4}
\end{equation}
for some $\consMultiplicityCoeff{\beta}>0$ depending only on $\Gamma$
and $\beta$.
\end{proof}

\subsection{Lower bound for the multiplicity of new eigenvalues\label{subsec:lower_bound}}

Let $\Upsilon$ be a nonelementary, finitely generated subgroup of
$\SL(2,\ZZ)$ and let $S=\Upsilon\backslash\Hyp$. In this subsection
we establish a lower bound for the multiplicity of new eigenvalues---defined
below---of a congruence covering $S_{n}$ of $S$ using the representation
theory of $\Upsilon/\Upsilon_{n}$. This well-known strategy was popularized
by the work \cite{sarnak_bounds_1991} of Sarnak and Xue.

The quotient group $\Upsilon/\Upsilon_{n}$ acts on the left on $S_{n}=\Upsilon_{n}\backslash\Hyp$
by

\[
(\gamma\Upsilon_{n})\cdot\Upsilon_{n}x=\gamma\Upsilon_{n}x.
\]
The Laplace-Beltrami operator of $S_{n}$ commutes with $\Upsilon/\Upsilon_{n}\curvearrowright S_{n}$,
hence $\Upsilon/\Upsilon_{n}$ acts on any eigenspace of $\Delta_{S_{n}}$.
Since $\Upsilon$ is nonelementary, it is Zariski-dense in $\SLne(2)$.
A result of Matthews, Vaserstein and Weisfeiler \cite{matthews-vaserstein-weisfeiler_congruence_1984}
implies then that $\Upsilon/\Upsilon_{n}$ is isomorphic to $\SL(2,\ZZ/n\ZZ)$
for most $n$.

\global\long\def\consGammaFillsSL#1{N_{#1}}%

\begin{lem}
\label{prop:GammaFillsSL2p}Let $\Upsilon$ be a finitely generated,
nonelementary subgroup of $\SL(2,\ZZ)$. There is an integer $N_{\Upsilon}$
with the next property: for any $n$ relatively prime to $N_{\Upsilon}$,
the projection $\Upsilon\to\SL(2,\ZZ/n\ZZ)$ is surjective.
\end{lem}

Consider an eigenvalue $\lambda$ of $S_{n}$. We denote $E_{n}(\lambda)$
the $\lambda$-eigenspace of $\Delta_{S_{n}}$ and let $E_{d}^{n}(\lambda)$
be the lift of $E_{d}(\lambda)$ to $E_{n}(\lambda)$ whenever $d$
divides $n$. We say that $\lambda$ is an \emph{old eigenvalue }of
$S_{n}$ if and only if any $\varphi\in E_{n}(\lambda)$ is a sum
$\sum_{d}\varphi_{d}$, with $d$ running in the divisors of $n$
in $[1,n)$, and $\varphi_{d}\in E_{d}^{n}(\lambda)$. Otherwise $\lambda$
is a \emph{new eigenvalue. }Let us give an equivalent definition of
new eigenvalue. Express $E_{n}(\lambda)$ as orthogonal sum of irreducible,
$\Upsilon/\Upsilon_{n}$-invariant subspaces $\mathcal{H}_{1}\widehat{\oplus}\cdots\widehat{\oplus}\mathcal{H}_{k}$.
Note that $E_{d}^{n}(\lambda)$ is the subspace of vectors of $E_{n}(\lambda)$
fixed by the normal subgroup $\Upsilon_{d}/\Upsilon_{n}$ of $\Upsilon/\Upsilon_{n}$.
Hence each nonzero $E_{d}^{n}(\lambda)$ is a sum of $\mathcal{H}_{j}$'s.
Thus $\lambda$ is new if and only if $E_{n}(\lambda)$ contains an
irreducible representation of $\Upsilon/\Upsilon_{n}$ not factoring
through\footnote{Let $N$ be a normal subgroup of a group $G$ and let $(\sigma,V)$
be a representation of $G$. The space $V^{N}$ of $N$-invariant
vectors of $V$ is $G$-invariant. Hence either $V^{N}=0$ or $\sigma$
factors through $G/N$.} $\Upsilon/\Upsilon_{d}$ for any divisor $d\in[1,n)$ of $n$.

We denote by $\omega(n)$ the number of prime divisors of an integer
$n$.
\begin{lem}
\label{lem:bound_reps_SL2n}Let $n>1$ be an integer all of whose
prime divisors are $\geq5$. The dimension of an irreducible complex
representation of $\SL(2,\ZZ/n\ZZ)$ not factoring through $\SL(2,\ZZ/d\ZZ)$,
for any divisor $d\in[1,n)$ of $n$, is $>\frac{n}{3^{\omega(n)}}$.
\end{lem}

\begin{proof}
Consider $n$ as in the statement and let $p_{1}^{a_{1}}\cdots p_{m}^{a_{m}}$
be its decomposition as product of primes. Since $\SL(2,\ZZ/n\ZZ)$
is isomorphic to 
\[
\prod_{j=1}^{m}\SL(2,\ZZ/p_{j}^{a_{j}}\ZZ),
\]
any irreducible representation $\sigma$ of $\SL(2,\ZZ/n\ZZ)$ is
a tensor product of irreducible representations $\sigma_{j}$ of the
$\SL(2,\ZZ/p_{j}^{a_{j}}\ZZ)$'s. For any $\sigma$ as in the statement,
the components $\sigma_{j}$ do not factor through $\SL(2,\ZZ/p_{j}^{a_{j}-1}\ZZ)$.
Then we know by \cite[Ex. 4.7.3, p. 245]{kowalski_introduction_2014}
and \cite[Lemma 7.1]{bourgain-gamburd_expansion_2008} that 

\[
\mathrm{dim}\,\sigma_{j}\geq\begin{cases}
\frac{1}{2}(p_{j}-1) & \text{when }a_{j}=1,\\
\frac{1}{2}(p_{j}^{a_{j}}-p_{j}^{a_{j}-2}) & \text{when }a_{j}\geq2.
\end{cases}
\]
In both cases $\mathrm{dim}\,\sigma_{j}>p^{a_{j}}/3$, so the result
follows.
\end{proof}
The lower bound for the multiplicity of new eigenvalues of $S_{n}$
follows immediately from lemmas \ref{prop:GammaFillsSL2p} and \ref{lem:bound_reps_SL2n}.
\begin{cor}
\label{prop:NewEigenvaluesHaveHighMultiplicity}Consider a finitely
generated, non-elementary subgroup $\Upsilon$ of $\SL(2,\ZZ)$ and
$S=\Upsilon\backslash\Hyp$. There is an integer $\consGammaFillsSL{\Upsilon}$
such that for any $n>1$ relatively prime to $\consGammaFillsSL{\Upsilon}$,
any new eigenvalue of $S_{n}$ has multiplicity $>\frac{n}{3^{\omega(n)}}$.
\end{cor}

\subsection{Proof of Theorem \ref{thm:MainThm}\label{subsec:main_proof}}

Let us now complete the proof of our main result.
\begin{proof}
[Proof of Theorem \ref{thm:MainThm}] We fix some $\beta\in(t_{\Gamma},\gexp)$.
Consider constants $\consMultiplicityp,\consMultiplicityCoeff{\beta}>0$
and $\xi=\xi(\gexp,\beta)\in(0,1)$ as in Proposition \ref{prop:Multiplicity_bound},
and $\consGammaFillsSL{\Gamma}$ as in Corollary \ref{prop:NewEigenvaluesHaveHighMultiplicity}.
Fix $C(\Gamma,\beta)\geq\max\{\consMultiplicityp,\consGammaFillsSL{\Gamma},5\}$
such that 
\begin{equation}
\frac{C(\Gamma,\beta)^{1-\xi}}{3}>\max\{1,\consMultiplicityCoeff{\beta}\}.\label{eq:FIN}
\end{equation}
Let $n$ be an integer with all its prime divisors $\geq C(\Gamma,\beta)$.
It suffices to show that all the eigenvalues $\lambda$ of $S_{n}$
in the interval $J_{\beta}=[0,\beta(1-\beta)]$ are old for any such
$n$\footnote{Indeed, let us show that this implies $E_{n}(\lambda)=E_{1}^{n}(\lambda)$
by induction on the number $\Omega(n)$ of prime divisors of $n$
counted with multiplicity: When $n$ is prime, this is the definition
of $\lambda$ being old. For the inductive step, the primes dividing
any divisor $d\in(1,n)$ of $n$ are still $\geq C(\Gamma,\beta)$
and $\Omega(d)<\Omega(n)$, so $E_{d}(\lambda)=E_{1}^{d}(\lambda)$.
Thus $E_{d}^{n}(\lambda)=E_{1}^{n}(\lambda)$, and since $\lambda$
is old,

\[
E_{n}(\lambda)=\sum_{\substack{d\mid n\\
d<n
}
}E_{d}^{n}(\lambda)=E_{1}^{n}(\lambda).
\]
}. 

We proceed by contradiction. Suppose $S_{n}$ has a new eigenvalue
$\lambda\in J_{\beta}$. Write $n$ as product of primes 
\[
n=p_{1}^{a_{1}}\cdots p_{k}^{a_{k}}.
\]
Since $n$ is relatively prime to $\consGammaFillsSL{\Gamma}$, then
\[
m_{\Gamma}(n,\lambda)>\frac{n}{3^{k}}
\]
by Corollary \ref{prop:NewEigenvaluesHaveHighMultiplicity}. Since
$\gexp>\frac{4}{5}$ and $n>\consMultiplicityp$, by Proposition \ref{prop:Multiplicity_bound}
we have 

\[
m_{\Gamma}(n,J_{\beta})\leq\consMultiplicityCoeff{\beta}n^{\xi}.
\]
Hence
\[
\consMultiplicityCoeff{\beta}>\frac{n^{1-\xi}}{3^{k}}\geq\frac{p_{1}^{1-\xi}}{3}\geq\frac{C(\Gamma,\beta)^{1-\xi}}{3},
\]
which contradicts (\ref{eq:FIN}).
\end{proof}
\bibliographystyle{amsalpha}
\bibliography{biblio}

\hspace{10pt} 

Irving Calderón

Department of Mathematical Sciences, 

Durham University, 

Lower Mountjoy, 

DH1 3LE Durham, 

United Kingdom

{\tt irving.d.calderon-camacho@durham.ac.uk}
\\

Michael Magee

Department of Mathematical Sciences, 

Durham University, 

Lower Mountjoy, 

DH1 3LE Durham, 

United Kingdom

{\tt michael.r.magee@durham.ac.uk}
\end{document}